\documentclass[12pt,a4paper,leqno,noamsfonts]{amsart}
\usepackage[english]{babel}
\usepackage[dvipsnames]{xcolor}
\usepackage{graphicx,pifont,soul}
\usepackage[utopia]{mathdesign}
\usepackage[a4paper,top=3cm,bottom=3cm,
           left=3cm,right=3cm,marginparwidth=60pt]{geometry}
\usepackage[utf8]{inputenc}
\usepackage{braket,caption,comment,mathtools,stmaryrd,physics}
\usepackage[usestackEOL]{stackengine}
\usepackage{multirow,booktabs,microtype,relsize}
\usepackage[colorlinks,bookmarks]{hyperref} %
      \hypersetup{colorlinks,%
            citecolor=britishracinggreen,%
            filecolor=black,%
            linkcolor=cobalt,%
            urlcolor=cornellred}
      \numberwithin{equation}{section}
\usepackage[capitalise]{cleveref}
\DeclareSymbolFont{cmarrows}{OMS}{cmsy}{m}{n}
\SetSymbolFont{cmarrows}{bold}{OMS}{cmsy}{b}{n}
\DeclareMathSymbol{\cmminus}{\mathbin}{cmarrows}{"00}

\DeclareMathSymbol{\leftrightarrow}{\mathrel}{cmarrows}{"24}
\DeclareMathSymbol{\leftarrow}{\mathrel}{cmarrows}{"20}
   
\DeclareMathSymbol{\rightarrow}{\mathrel}{cmarrows}{"21}
   \let\to=\rightarrow
\DeclareMathSymbol{\mapstochar}{\mathrel}{cmarrows}{"37}
   \def\mapsto{\mapstochar\rightarrow}
   
   \DeclareSymbolFont{usualmathcal}{OMS}{cmsy}{m}{n}
\DeclareSymbolFontAlphabet{\mathcal}{usualmathcal}
\makeatletter
\newcommand{\mylabel}[2]{#2\def\@currentlabel{#2}\label{#1}}
\makeatother
\usepackage[colorinlistoftodos]{todonotes}

\definecolor{cornellred}{rgb}{0.7, 0.11, 0.11}
\definecolor{britishracinggreen}{rgb}{0.0, 0.26, 0.15}
\definecolor{cobalt}{rgb}{0.0, 0.28, 0.67}

\newcommand{\BA}{{\mathbb{A}}}

\newcommand{\BG}{{\mathbb{G}}}

\newcommand{\BL}{{\mathbb{L}}}

\newcommand{\BN}{{\mathbb{N}}}

\newcommand{\BP}{{\mathbb{P}}}
\newcommand{\BQ}{{\mathbb{Q}}}

\newcommand{\BZ}{{\mathbb{Z}}}

\newcommand{\CA}{{\mathcal A}}

\newcommand{\CE}{{\mathcal E}}
\newcommand{\CF}{{\mathcal F}}

\newcommand{\CI}{{\mathcal I}}

\newcommand{\CL}{{\mathcal L}}

\newcommand{\CR}{{\mathcal R}}

\newcommand{\CU}{{\mathcal U}}
\newcommand{\CV}{{\mathcal V}}

\newcommand{\CX}{{\mathcal X}}
\newcommand{\CY}{{\mathcal Y}}
\newcommand{\CZ}{{\mathcal Z}}

\newcommand{\OO}{\curly O}

\newcommand{\cF}{\curly F}
\newcommand{\cG}{\curly G}

\newcommand{\id}{\operatorname{id}}
\newcommand{\Spec}{\operatorname{Spec}}

\newcommand{\Supp}{\operatorname{Supp}}

\newcommand{\simto}{\,\widetilde{\to}\,}
\newcommand{\into}{\hookrightarrow}
\newcommand{\onto}{\twoheadrightarrow}
\newcommand{\bfk}{\mathbf{k}}
\newcommand{\opp}{\mathsf{op}}
\newcommand{\sm}{\mathsf{sm}}
\newcommand{\inj}{\mathsf{inj}}
\newcommand{\red}{\mathsf{red}}
\newcommand{\aff}{\mathsf{aff}}

\newcommand{\supp}{\mathsf{supp}}

\newcommand{\FS}{\mathfrak{S}}
\newcommand{\HH}{\mathrm{H}}

\newcommand{\XX}{\mathrm{x}}
\newcommand{\V}{\mathrm{V}}
\newcommand{\Iso}{\mathbf{Isom}}
\newcommand{\CCoh}{\mathscr{C}\kern-0.25em {o}\kern-0.2em{h}}

\DeclareMathOperator{\Aut}{Aut}
\DeclareMathOperator{\cycle}{cycle}
\DeclareMathOperator{\Ten}{T}
\DeclareMathOperator{\Pol}{Pol}
\DeclareMathOperator{\TS}{TS}
\DeclareMathOperator{\Alg}{Alg}

\DeclareMathOperator{\Hilb}{Hilb}
\DeclareMathOperator{\Sets}{Sets}
\DeclareMathOperator{\Sch}{Sch}
\DeclareMathOperator{\Quot}{Quot}

\DeclareMathOperator{\Ann}{Ann}
\DeclareMathOperator{\Sym}{Sym}
\DeclareMathOperator{\Coh}{Coh}
\DeclareMathOperator{\QCoh}{QCoh}
\DeclareMathOperator{\Ob}{Ob}
\DeclareMathOperator{\GL}{GL}

\DeclareMathOperator{\rk}{rk}
\DeclareMathOperator{\St}{St}
\DeclareMathOperator{\length}{length}
\DeclareMathOperator{\mult}{mult}
\DeclareMathOperator{\Exp}{Exp}
\DeclareMathOperator{\Var}{Var}
\DeclareMathOperator{\pr}{pr}

\DeclareFontFamily{OT1}{rsfs}{}
\DeclareFontShape{OT1}{rsfs}{n}{it}{<-> rsfs10}{}
\DeclareMathAlphabet{\curly}{OT1}{rsfs}{n}{it}

\newcommand\Hom{\operatorname{Hom}}
\newcommand{\lHom}{\mathscr{H}\kern-0.3em {o}\kern-0.2em{m}}
\newcommand{\RHom}{{\mathbf{R}\kern-0.07em\mathrm{Hom}}}
\newcommand{\RRlHom}{\mathbf{R}\kern-0.025em\mathscr{H}\kern-0.3em {o}\kern-0.2em{m}}

\newcommand{\lExt}{{\mathscr{E}\kern-0.2em xt}}
\newcommand{\End}{\operatorname{End}}
\newcommand{\lEnd}{\mathscr{E}\kern-0.15em {n}\kern-0.1em{d}}

\usepackage[all]{xy}
\usepackage{tikz}
\usepackage{tikz-cd}
\usepackage{adjustbox}
\usepackage{rotating}
\usepackage{comment}

\usetikzlibrary{matrix,shapes,intersections,arrows,decorations.pathmorphing}
\tikzset{commutative diagrams/.cd,
mysymbol/.style={start anchor=center,end anchor=center,draw=none}}
\newcommand\MySymb[2][\square]{%
  \arrow[mysymbol]{#2}[description]{#1}}
\tikzset{
shift up/.style={
to path={([yshift=#1]\tikztostart.east) -- ([yshift=#1]\tikztotarget.west) \tikztonodes}
}
}

\newcommand*{\defeq}{\mathrel{\vcenter{\baselineskip0.5ex \lineskiplimit0pt
                     \hbox{\scriptsize.}\hbox{\scriptsize.}}}%
                     =}
                     
\theoremstyle{definition}

\newtheorem*{lemma*}{Lemma}
\newtheorem*{theorem*}{Theorem}
\newtheorem*{example*}{Example}
\newtheorem*{fact*}{Fact}
\newtheorem*{notation*}{Notation}
\newtheorem*{definition*}{Definition}
\newtheorem*{prop*}{Proposition}
\newtheorem*{remark*}{Remark}
\newtheorem*{corollary*}{Corollary}
\newtheorem*{conventions*}{Conventions}

\newtheorem{definition}{Definition}[section]

\newtheorem{openproblem}[definition]{Open Problem}
\newtheorem{example}[definition]{Example}

\newtheorem{notation}[definition]{Notation}
\newtheorem{remark}[definition]{Remark}

\newtheoremstyle{thm} % <name> % (ambienti con dimostrazione)
        {3mm}% <Space above>
        {3mm}% <Space below>
        {\slshape}% <Body font> % 
        {0mm}% <Indent amount>
        {\bfseries}% <Theorem head font>
        {.}% <Punctuation after theorem head>
        {1mm}% <Space after theorem head>
        {}% <Theorem head spec (can be left empty, meaning 'normal')> 
\theoremstyle{thm}
\newtheorem{theorem}[definition]{Theorem}
\newtheorem{corollary}[definition]{Corollary}
\newtheorem{lemma}[definition]{Lemma}
\newtheorem{prop}[definition]{Proposition}
\newtheorem{thm}{Theorem}



\DeclareMathOperator{\Fitting}{Fitt}
\newcommand{\Ab}{\mathbb A}

\newcommand{\Lb}{\mathbb L}

\title[On the stack of 0-dimensional coherent sheaves: motivic aspects]{On the stack of 0-dimensional coherent sheaves: \\ motivic aspects}
\author{Barbara Fantechi, Andrea T. Ricolfi}

\begin{document}
\maketitle
\begin{abstract}
Let $X$ be a variety. In this survey, we study (decompositions of) the motivic class, in the Grothendieck ring of stacks, of the stack $\CCoh^n(X)$ of $0$-dimensional coherent sheaves of length $n$ on $X$. To do so, we review the construction of the support map $\CCoh^n(X) \to \Sym^n(X)$ to the symmetric product and we prove that, for any closed point $p \in X$, the motive of the punctual stack $\CCoh^n(X)_p$ parametrising sheaves supported at $p$ only depends on a formal neighbourhood of $p$. We perform the same analysis for the Quot-to-Chow morphism $\Quot_X(\CE,n) \to \Sym^n(X)$, for a fixed sheaf $\CE \in \Coh X$.
\end{abstract}

{\hypersetup{linkcolor=black}\tableofcontents}

\section{Introduction}
Let $X$ be an algebraic variety defined over an algebraically closed field $\bfk$ of characteristic $0$. Consider the algebraic stack 
\[
\CCoh^n(X)
\]
parametrising coherent sheaves $\CF \in \Coh(X)$ such that
\begin{equation}\label{eqn:0-dim-sheaves}
\dim \Supp(\CF) = 0, \,\,\,\chi(\CF) = n.
\end{equation}
Some structural aspects of this stack have been treated in \cite{Fantechi-Ricolfi-structural}. In this paper we mostly focus on \emph{motivic} aspects, i.e.~on the class of $\CCoh^n(X)$ in the Grothendieck ring $K_0(\St_{\bfk}^{\aff})$ of algebraic stacks (cf.~\Cref{def:K-stacks}). Some of the results contained here are well-known to experts, but some of the proofs are hard to find in the literature; we have therefore decided to include them, to provide a friendly introduction to the subject.

Our viewpoint is the basic observation that a stratification of $\CCoh^n(X)$ induces a stratification on each moduli space mapping into $\CCoh^n(X)$. An example is provided by the \emph{Quot scheme of points} $\Quot_X(\CE,n)$, parametrising $0$-dimensional length $n$ quotients of a fixed coherent sheaf $\CE$ over $X$. For instance, if $\CE=\OO_X$, one recovers the \emph{Hilbert scheme of points} $\Hilb^n(X)$, parametrising $0$-dimensional subschemes $Z \into X$ of length $n$. Therefore, by picking a `clever' stratification of $\CCoh^n(X)$, one might hope to compute the motive of the Quot scheme of points by pullback.

Ideally, one would like to determine, as explicitly as possible, the motivic generating functions
\begin{equation}\label{eqn:motivic-series}
\begin{split}
\mathsf Z_X(t) &= \sum_{n\geqslant 0}\,[\CCoh^n(X)] t^n \,\in\, K_0(\St_{\bfk}^{\aff})\llbracket t\rrbracket\\
\mathsf Q_\CE(t) &= \sum_{n\geqslant 0}\,[\Quot_X(\CE,n)] t^n\,\in\, K_0(\Var_{\bfk})\llbracket t\rrbracket
\end{split}
\end{equation}
at least when $X$ is a smooth $\bfk$-variety and (for the second series) $\CE$ is a locally free sheaf over $X$. By standard power structure arguments, it is enough to deal with $\BA^{\dim X}$. As far as we know, this is currently out of reach for $\dim X > 2$. We present the known results for $\dim X \leqslant 2$ in \Cref{sec:state-of-art}. The reader can also consult \cite{double-nested-1,MOTIVES,MR_nested_Quot,MR-hyperquot,Motives-nested} for related motivic computations. The purpose of this paper is to lay down the technical foundations for tackling this kind of computations.

\smallbreak
We now briefly discuss the main contents of this work.

In \Cref{sec:background} we provide minimal background on families of sheaves, notions of support and motivic rings, such as $K_0(\Var_\bfk)$ and $K_0(\St^{\aff}_{\bfk})$. In \Cref{sec:coh(X/B)} we present the definition of the $B$-stack $\CCoh(X/B) \to B$ associated to a quasiprojective morphism $f\colon X \to B$ to a locally noetherian scheme $B$. 
We show in \Cref{thm:coh-is-algebraic} that $\CCoh(X/B)$ is algebraic following \cite{LMB}, thus including a few more assumptions, namely that $f$ is projective and $f_\ast \OO_X=\OO_B$ holds universally. Under these assumptions, we are able to underline the key role of Grothendieck's Quot scheme in the construction of an explicit smooth atlas $\mathrm{Q} \to\CCoh(X/B)$. Moreover, the stack is shown to have \emph{affine geometric stabilisers} (see \Cref{thm:coh-is-algebraic} and \Cref{lemma:affine_stab}). This implies that, if $X$ is a $\bfk$-variety, the open substacks (cf.~\Cref{sec:0-dim}) 
\[
\CCoh^n(X) \subset \CCoh(X/\bfk),
\]
parametrising 0-dimensional length $n$ sheaves as in \eqref{eqn:0-dim-sheaves}, carry a well-defined motivic class
\[
[\CCoh^n(X)] \,\in\,K_0(\St^{\aff}_{\bfk}),
\]
making the first series in \eqref{eqn:motivic-series} well-defined.

Following Rydh \cite{Rydh1}, we provide in \Cref{sec:supp-map} (and start using in \Cref{sec:coh-to-sym}) the construction of the \emph{Coh-to-Chow morphism} (or \emph{support map})
\begin{equation}\label{coh-to-sym-intro}
\begin{tikzcd}[row sep=tiny]
\CCoh^n(X) \arrow{rr}{\supp_X^n} & & \Sym^n(X) \\
{[\CF]} \arrow[mapsto]{rr} & & \displaystyle\sum_{x \in X} \length_{\OO_{X,x}}\CF_x\cdot x
\end{tikzcd}
\end{equation}
to the symmetric product of $\Sym^n(X) = X^n / \FS_n$.

\smallbreak
Given a coherent sheaf $\CE \in \Coh(X)$, the natural morphism (cf.~\Cref{lemma:quot-to-coh})
\[
\begin{tikzcd}[row sep=tiny]
\Quot_X(\CE,n) \arrow{r}{\rho_{\CE,n}} &  \CCoh^n(X) \\   
{[}\CE \onto \CF{]} \arrow[mapsto]{r} & {[}\CF{]}
\end{tikzcd} 
\]
can be composed with $\supp_X^n$ to get the \emph{Quot-to-Chow morphism} 
\[
\begin{tikzcd}
\Quot_X(\CE,n) \arrow{rr}{\mathsf{qc}_{\CE,n}} &&  \Sym^n(X).
\end{tikzcd}
\]
For a locally closed subscheme $Z \into X$, we define $\CCoh^n(X)_Z \into \CCoh^n(X)$ to be the preimage of $\Sym^n(Z)$ along $\supp_X^n$. Similarly, $\Quot_X(\CE,n)_Z\into\Quot_X(\CE,n)$ is defined to be the preimage of $\Sym^n(Z)$ along $\mathsf{qc}_{\CE,n}$. 

Yet another schematic structure on the locus of length $n$ sheaves (or sheaf quotients) with set-theoretic support $Z$ is presented in \Cref{sec:support-on-subset}. This structure is defined as
\begin{equation}
    \begin{split}
 \CCoh^n(X,Z) &\defeq \CCoh^n(Z_n), \\
 \Quot_X(\OO_X^{\oplus r},n,Z) &\defeq \Quot_{Z_n}(\OO_{Z_n}^{\oplus r},n),
    \end{split}
\end{equation}
where $Z_n \into X$ is the locally closed subscheme defined by $\mathscr I_Z^n$, where $\mathscr I_Z \subset \OO_U$ is the ideal sheaf of $Z$ in an open subset $U$ of $X$ containing $Z$.
There are isomorphisms $\CCoh^n(X,Z)_{\red}\simto (\CCoh^n(X)_Z)_{\red}$ and $\Quot_X(\OO_X^{\oplus r},n,Z)_{\red} \simto (\Quot_X(\OO_X^{\oplus r},n)_Z)_{\red}$ inducing the equalities of motivic classes 
\begin{equation}
\label{eqn:id-Z}
\begin{split}
[\CCoh^n(X,Z)] &= [\CCoh^n(X)_Z],\\
[\Quot_X(\OO_X^{\oplus r},n,Z)] &= [\Quot_X(\OO_X^{\oplus r},n)_Z].
\end{split}
\end{equation}

We shall prove the following motivic decompositions. 

\begin{thm}[\Cref{prop:coh-dec}, \Cref{cor:quot-dec}] \label{thm:motivic-dec-intro}
Let $X$ be a $\bfk$-variety,  $Z\into X$ a closed subscheme with complement $U = X \setminus Z$, and $n \in \BZ_{\geqslant 0}$ an integer. 
\begin{itemize}
\item [\mylabel{main1-1}{\normalfont{(1)}}] There is an identity 
\[
[\CCoh^n(X)]=\sum_{0\leqslant m\leqslant n} \,[\CCoh^m(U)] [\CCoh^{n-m}(X)_Z] \,\in\,K_0(\St^{\aff}_{\bfk}).
\]
\item [\mylabel{main1-2}{\normalfont{(2)}}] Given a coherent sheaf $\CE \in \Coh(X)$, there is an identity
\[
[\Quot_X(\CE,n)] = \sum_{0\leqslant m\leqslant n} \,[\Quot_U(\CE|_U,m)] [\Quot_X(\CE,n-m)_Z]\,\in\,K_0(\Var_{\bfk}).
\]
\end{itemize}
\end{thm}

The case where $Z$ is a closed point $p \in X$ is of particular interest. We present a proof ot the following result.

\begin{thm}[{\Cref{thm:punctual-smooth-point}, \Cref{thm:quot-punctual}}]
\label{main1:punctual-stack}
Let $X$ be a $\bfk$-variety, $p \in X$ a smooth closed point with $\dim T_pX=d$, and $n \in \BZ_{\geqslant 0}$ an integer. 

\begin{itemize}
\item [\mylabel{PP1}{(1)}] There is an isomorphism of algebraic stacks
\[
\begin{tikzcd}
    \CCoh^n(X)_p \arrow{r}{\sim} & \CCoh^n(\BA^d)_0.
\end{tikzcd}
\]
\item [\mylabel{PP2}{(2)}] For every $r>0$ there is an isomorphism of schemes
\[
\begin{tikzcd}
\Quot_{X}(\OO_X^{\oplus r},n)_p  \arrow{r}{\sim} & \Quot_{\BA^d}(\OO_{\BA^d}^{\oplus r},n)_0.
\end{tikzcd}
\]
\end{itemize}
\end{thm}

Let $\sigma$ be a singularity type, i.e.~an equivalence class of pointed schemes $(Y,y)$ under the equivalence relation `having the same completed local ring' (\Cref{def:sing-type}). Because of the identities \eqref{eqn:id-Z} and the canonical isomorphisms $\OO_{X,p}/\mathfrak m^n_p \simto \widehat{\OO}_{X,p}/\widehat{\mathfrak m}^n_p$, the generating functions
\begin{align*}
\mathsf Z_{\sigma}(t) 
&= \sum_{n\geqslant 0} \,[\CCoh^n(Y)_y] t^n \\
\mathsf Q_{r,\sigma}(t) 
&= \sum_{n\geqslant 0}\,[\Quot_Y(\OO_Y^{\oplus r},n)_y] t^n
\end{align*}
are independent on the representative $(Y,y)$ of $\sigma$. If $Y$ is smooth of dimension $d$ at $y \in Y$, by \Cref{main1:punctual-stack} the corresponding generating functions are
\begin{align*}
\mathsf Z_{\sm_d}(t) &= \sum_{n\geqslant 0} \,[\CCoh^n(\BA^d)_0] t^n \\
\mathsf Q_{r,\sm_d}(t) &= \sum_{n\geqslant 0}\,[\Quot_{\BA^d}(\OO_{\BA^d}^{\oplus r},n)_0]t^n.
\end{align*}
We exploit this observation in order to present general formulas for the generating functions \eqref{eqn:motivic-series}. More precisely, we prove the following identities (partially in the language of power structures, recalled in \Cref{subsec:power-str}).

\begin{thm}[{\Cref{thm:gen-fct-coh}, \Cref{thm:Quot-Series}}]\label{main3}
Let $X$ be a $\bfk$-variety of dimension $d$, with smooth locus $X_{\sm} \subset X$. Let $\sigma_1,\ldots,\sigma_e$ be distinct singularity types, and assume $X$ has $k_i$ points of singularity type $\sigma_i$ for $i=1,\ldots,e$, and no other singularities. 
\begin{enumerate}
\item There are product decompositions
\[
\mathsf Z_X(t) = \mathsf Z_{X_{\sm}}(t)  \prod_{1 \leqslant i \leqslant e} \mathsf Z_{\sigma_i}(t)^{k_i} = \mathsf Z_{\sm_d}(t)^{[X_{\sm}]} \prod_{1 \leqslant i \leqslant e} \mathsf Z_{\sigma_i}(t)^{k_i}.
\]
\item If $\CE$ is a locally free sheaf of rank $r>0$ on $X$, there are product decompositions
\[
\mathsf Q_\CE(t) = \mathsf Q_{\CE|_{X_{\sm}}}(t)  \prod_{1\leqslant i\leqslant e} \mathsf Q_{r,\sigma_i}(t)^{k_i} = \mathsf Q_{r,\sm_d}(t)^{[X_{\sm}]}  \prod_{1\leqslant i\leqslant e} \mathsf Q_{r,\sigma_i}(t)^{k_i}.
\]
\end{enumerate}
\end{thm}

We do not currently know of an analogue of \Cref{main3} for the case where $X$ has singular locus of positive dimension. We leave this for future investigation.

\subsection*{Acknowledgements} 
The authors are partially supported by the PRIN \emph{Geometry of algebraic structures: moduli, invariants, deformations} 2022BTA242. Both authors are members of GNSAGA of INDAM. We thank J.~Jelisiejew for pointing out a mistake in a previous version of this work, and the anonymous referees for their useful comments.

%%%%%%%%%%%%%%%%%%%%%%%%%%%%%%%%%%%%%%%%%%%%%%%%%%%%%%%%%%%%%%%%%%%%%%%%%%%%%%%%%%%%%%%%%%%%%
%%%%%%%%%%%%%%%%%%%%%%%%%%%%%%%%%%%%%%%%%%%%%%%%%%%%%%%%%%%%%%%%%%%%%%%%%%%%%%%%%%%%%%%%%%%%%

\section{Conventions and background material}\label{sec:background}
In this section we state our conventions and we provide minimal background on families of coherent sheaves, and on the Grothendieck ring of stacks. Some basic theory of algebraic stacks is reviewed in \Cref{sec:stacks}.

\subsection{Conventions}
All schemes in this paper are locally noetherian. A morphism of schemes $f\colon X \to B$ is quasiprojective (resp.~projective) if it is of finite type (resp.~proper) and there is an $f$-ample invertible sheaf on $X$ \cite[\href{https://stacks.math.columbia.edu/tag/01VV}{Tag 01VV}]{stacks-project}. A \emph{variety} will be an integral scheme quasiprojective over an algebraically closed field $\bfk$ of characteristic $0$, that is kept fixed throughout. We simply write `$X \times T$' for fibre products over $\Spec \bfk$. For a scheme $Y$, we denote by $\Coh(Y)$, resp.~$\QCoh(Y)$, the abelian category of coherent, resp.~quasicoherent sheaves on $Y$. 

Given a morphism of schemes $Y \to T$, and a quasicoherent sheaf $\cF$ on $Y$, we say that $\cF$ is \emph{flat over} $T$, or $T$-\emph{flat}, if for any point $y \in Y$, with image $t \in T$, the stalk $\cF_y$ is flat as a module over $\OO_{T,t}$, via the canonical ring homomorphism $\OO_{T,t}\to \OO_{Y,y}$. Given a point $t \in T$, we denote by $\cF_t$ the restriction $\cF|_{Y_t}$  of $\cF$ to the fibre $Y_t = Y \times_T \Spec \kappa(t)$.

By the word `stratification', we mean the following: given a scheme $Y$ and locally closed subschemes $Z_i \into Y$, we say that they form a stratification of $Y$ if the induced morphism $f\colon \coprod_i Z_i \to Y$ is bijective. We use the same notion for stacks, requiring $f$ to induce an equivalence on geometric points (cf.~\Cref{def:stratification}).

\subsection{Coherent sheaves and notions of support}
Recall (see e.g.~\cite[\href{https://stacks.math.columbia.edu/tag/01XZ}{Tag 01XZ}]{stacks-project}) that, on a locally noetherian scheme $Y$, an $\OO_Y$-module $\cF$ is coherent if and only if it is quasicoherent and of finite type, if and only if it is finitely presented, which means (cf.~\cite[\href{https://stacks.math.columbia.edu/tag/01BN}{Tag 01BN}]{stacks-project}) that every point $y \in Y$ has an open neighbourhood $V \subset Y$ such that $\cF|_V$ is the cokernel of a map of free $\OO_V$-modules of finite rank. Recall, also, that the pullback of a coherent sheaf, along a morphism of locally noetherian schemes, is still coherent. Therefore, if $X \to B$ is quasiprojective (with $B$ locally noetherian), $\cF \in \Coh(X)$ and $B'$ is a locally noetherian $B$-scheme, then the pullback of $\cF$ along $X \times_BB' \to X$ is coherent, since both $X$ and $X \times_BB'$ are locally noetherian.

\begin{definition} Let $Y$ be a (locally noetherian) scheme and $\cF\in\QCoh(Y)$.
\begin{itemize}
\item [\mylabel{supp-i}{(i)}] The {\em set-theoretic support} of $\cF$ is the subset
\[
\Supp_{\text{set}}(\cF) = \set{y\in Y\,|\, \cF_y\neq 0} \subset Y.
\]
\item [\mylabel{supp-ii}{(ii)}] The {\em scheme-theoretic support} of $\cF$ is the closed subscheme
\[
\begin{tikzcd}
\Supp(\cF) \arrow[hook]{r} &  Y
\end{tikzcd}
\]
defined by the annihilator ideal sheaf $\Ann(\cF) \subset \OO_Y$, namely the kernel of the canonical $\OO_Y$-linear homomorphism $\OO_Y \to \lHom_{\OO_Y}(\cF,\cF)$.
\item [\mylabel{supp-iii}{(iii)}] The dimension of $\Supp(\cF)$ is called the \emph{dimension of} $\cF$ and is denoted $\dim \cF$.
\end{itemize}
\end{definition}

\begin{remark}\label{rmk:supp}
Let $Y$ be a (locally noetherian) scheme and $\cF\in\QCoh(Y)$.
\begin{itemize}
    \item [\mylabel{rmk1}{(1)}] Let $y\in Y$ be a point. If $\cF \in \Coh(Y)$, by Nakayama's Lemma, $y \in \Supp_{\text{set}}(\cF)$ if and only if $y\in \Supp(\cF)$. In particular, the set-theoretic support of a coherent sheaf is a closed subspace of $Y$. 
    \item [\mylabel{rmk2}{(2)}] By definition, $\Supp(\cF)$ is the smallest amongst all closed subschemes $\iota \colon Z \into Y$ such that the natural map $\cF \to \iota_\ast \iota^\ast \cF$ is an isomorphism. 
\end{itemize}
\end{remark}

\begin{remark}\label{rmk:ev-map-and-annihilation}
If $\CF$ is a $0$-dimensional coherent sheaf on a $\bfk$-variety $X$, the natural evaluation map $\HH^0(X,\CF) \otimes_{\bfk}\OO_X \to \CF$ is surjective. If $\chi(\CF)=n$ and $p \in \Supp(\CF)$, then $\mathfrak m_p^n\CF_p = 0$.
\end{remark}

Another notion of support can be given via Fitting ideals. We will not need the Fitting support in this paper, but it is worth recalling.

\begin{definition}
Let $\cF$ be a coherent sheaf on a scheme $Y$. Fix an integer $i \geqslant -1$. Pick an open cover of $Y$ by open subschemes $V \into Y$ such that $\cF$ admits a local presentation
\[
\begin{tikzcd}
    \OO_V^{\oplus m} \arrow{r}{a_V} & \OO_V^{\oplus n}\arrow{r} & \cF|_V \arrow{r} & 0.
\end{tikzcd}
\]
The $i$-th \emph{Fitting ideal} of $\cF$ is the ideal sheaf $\Fitting^i(\cF) \subset \OO_Y$ having, as local generators over $V$, the minors of size $n-i$ of the matrix determined by $a_V$.
\end{definition}

We have the following properties.
\begin{enumerate}
\item The ideal sheaf $\Fitting^i(\cF) \subset \OO_Y$ does not depend on the chosen presentation of $\cF$, so the definition makes sense.
\item We have a filtration of ideal sheaves
\[
0 = \Fitting^{-1}(\cF) \subset \Fitting^0(\cF) \subset \cdots \subset \OO_Y.
\]
\item The closed subscheme $Z_0 = \V(\Fitting^0(\cF)) \into Y$ cut out by $\Fitting^0(\cF)$ contains $\Supp(\cF)$ as a closed subscheme. It defines yet another scheme structure on $\Supp_{\mathrm{set}}(\cF)$, i.e.~one has an isomorphism $\Supp(\cF)_{\red} \simto (Z_0)_{\red}$, and hence an identity $\lvert Z_0 \rvert = \lvert \Supp(\cF) \rvert$ of topological spaces.
\item Fitting ideals are invariant under arbitrary base change \cite[\href{https://stacks.math.columbia.edu/tag/0C3C}{Tag 0C3C}]{stacks-project}. The scheme-theoretic support $\Supp(\cF)$ is in general not invariant under base change; it is invariant under flat base change \cite[\href{https://stacks.math.columbia.edu/tag/01U2}{Tag 01U2}]{stacks-project}.
\end{enumerate}   

\subsection{Quot schemes}
Let $f\colon X \to B$ be a quasiprojective morphism, $\OO_X(1)$ an $f$-very ample invertible sheaf. Fix a coherent sheaf $\CR\in\Coh(X)$ and a polynomial $P \in \BQ[z]$. Grothendieck's \emph{Quot scheme} $\Quot_{X/B}(\CR,P)$ is the (quasiprojective) $B$-scheme whose $T$-valued points, for any $B$-scheme $T$, are the isomorphism classes of surjections $\CR_T \onto \cF$ in $\Coh(X\times_BT)$, where 
\begin{itemize}
    \item $\cF$ is a $T$-flat sheaf and $\Supp(\cF) \to T$ is proper,
    \item the Hilbert polynomial of $\cF_t$ is $P$ for every $t\in T$, and
    \item two surjections are isomorphic if they have the same kernel.
\end{itemize}  
We have denoted by $\CR_T$ the pullback of $\CR$ along $\pr_X\colon X \times_B T \to X$. We refer the reader to \cite{Grothendieck_Quot,fga_explained} and the references therein for basic material on Quot schemes.

\subsection{Families of sheaves}
The following definition is key to the definition of the stack of coherent sheaves.
\begin{definition}\label{def:family-coh-sheaves}
Let $B$ be a locally noetherian scheme, $X\to B$ a quasiprojective morphism, $T$ a locally noetherian $B$-scheme. 
\begin{itemize}
\item [\mylabel{family-1}{(1)}] A \emph{family of coherent sheaves} on $X\to B$, parametrised by $T$ (also called a $T$-family), is a $T$-flat coherent sheaf
\[
\cF\in\Coh(X\times_BT)
\]
such that $\Supp(\cF) \to T$ is proper.
\item [\mylabel{family-2}{(2)}] A \emph{family of~~$0$-dimensional sheaves} on a $X \to B$, parametrised by $T$, is a family of coherent sheaves $\cF \in \Coh(X\times_B T)$ such that $\Supp(\cF)\to T$ is finite (i.e.~proper with finite fibres, see \cite[\href{https://stacks.math.columbia.edu/tag/02LS}{Tag 02LS}]{stacks-project}). 
\end{itemize}
\end{definition} 

\begin{remark}
Let $\cF$ be a $T$-family of coherent sheaves as in \Cref{def:family-coh-sheaves}.
\begin{enumerate}
\item [(a)] The properness condition on $\Supp(\cF) \to T$ is redundant as soon as $X \to B$ is itself proper, because $\Supp(\cF) \into X\times_BT$ is a closed immersion and $X\times_BT\to T$ is proper by base change.
\item [(b)] Properness only depends on the reduced structure, in particular $\Supp(\cF) \to T$ is proper if and only if $\V(\Fitting^0(\cF)) \to T$ is proper. 
\item [(c)] The support $\Supp(\cF)\to T$ need not be flat even though $\cF$ is $T$-flat. 
\end{enumerate}
\end{remark}

Let $X \to B$ and $T$ be as above. When a $T$-flat sheaf $\cF\in\Coh(X\times_B T)$ receives a surjection
\[
\begin{tikzcd}
\OO_{X\times_B T}\arrow[two heads]{r} & \cF
\end{tikzcd}
\]
from the structure sheaf, its kernel $\CI\subset \OO_{X\times T}$ defines a closed subscheme
\[
\begin{tikzcd}
\iota\colon Z\arrow[hook]{r} & X\times_B T,
\end{tikzcd}
\]
flat over $T$, such that $\cF = \iota_\ast\OO_Z$. In this situation, when $Z \to T$ is furthermore finite, the number
\[
n = \length Z_t = \chi(\OO_{Z_t}) = h^0(Z_t,\OO_{Z_t}) \in \BN
\]
is locally constant on $t \in T$, and we say that $\cF$ is a family of $0$-dimensional subschemes of length $n$ on the fibres of $X \to B$ (parametrised by $T$). Such families correspond to the $T$-valued points of (i.e.~$B$-morphisms from $T$ to) the \emph{Hilbert scheme of points} $\Hilb^n(X/B)$. We refer the reader to \cite[Chapter~5]{fga_explained} and the references therein for a construction of Hilbert schemes and, more generally, Quot schemes.

The following is a useful criterion for flatness, which shall be used throughout without mention.

\begin{prop}[{\cite[III, Thm.~9.9]{Hartshorne_AG}}]\label{flatness-criteria}
Let $f\colon Y \to T$ be a projective morphism of schemes, with $T$ noetherian. Let $\CL$ be an $f$-ample invertible sheaf on $Y$. Fix a coherent sheaf $\cF \in \Coh(Y)$. Consider the following conditions:
\begin{itemize}
    \item [\mylabel{flat-1}{(1)}] $\cF$ is $T$-flat.
    \item [\mylabel{flat-2}{(2)}] $f_{\ast} (\cF \otimes_{\OO_Y} \CL^{\otimes m})$ is locally free of finite rank for $m \gg 0$.
    \item [\mylabel{flat-3}{(3)}] The Hilbert polynomial
    \[
P_\CL(\cF_t,k) = \chi(\cF_t \otimes_{\OO_{Y_t}}\CL_t^{\otimes k})
    \]
    is locally constant as a function on $T$.
\end{itemize}
Then {\normalfont{\ref{flat-1}}} is equivalent to {\normalfont{\ref{flat-2}}} and implies {\normalfont{\ref{flat-3}}}. All three conditions are equivalent if $T$ is reduced.
\end{prop}

%%%%%%%%%%%%%%%%%%%%%%%%%%%%%%%%%%%%%%%%%%%%%%%%%%%%%%%%%%%%%%%%%%%%%%%%%%%%%%%%%%%%%%%%%%%%%
\subsection{Grothendieck rings of stacks}\label{sec:G-groups}
The definition of the Grothendieck ring of algebraic stacks was given by Ekedahl \cite{EKStacks}. The reader familiar with it can safely skip this section.

\subsubsection{Definition and ring structure}
Let $B$ be a $\bfk$-variety. The Grothendieck group of $B$-varieties is the free abelian group $K_0(\Var_B)$ generated by isomorphism classes $[X \to B]$ of finite type $B$-varieties, modulo the \emph{scissor relations}, namely 
\[
[f\colon X \to B] = [f|_{U}\colon U \to B] + [f|_{X\setminus U}\colon X\setminus U \to B]
\]
whenever $U \into X$ is an open $B$-subvariety. 

The Grothendieck group $K_0(\Var_B)$ becomes a ring via the operation
\begin{equation}\label{product-K-ring}
[X_1 \to B]\cdot [X_2 \to B] = [X_1\times_B X_2 \to B].
\end{equation}

We let $\BL = [\BA^1_B \to B] \in K_0(\Var_B)$ denote the \emph{Lefschetz motive} over $B$.

\begin{definition}[{Ekedahl \cite{EKStacks}}]\label{def:K-stacks}
Let $B$ be a $\bfk$-variety. The Grothendieck group of $B$-stacks is the free abelian group $K_0(\St_B^{\aff})$ generated by isomorphism classes $[\CX \to B]$ of locally finite type algebraic $B$-stacks with affine geometric stabilisers, modulo the scissor relations, and the \emph{local triviality relation}, namely
\[
[\CX \xrightarrow{f} B] = \BL^r\cdot  [\CY \xrightarrow{g} B]
\]
whenever $f$ factors as $g\circ \pi$ for $\pi\colon E \to \CY$ the projection from the total space of a rank $r$ vector bundle over $\CY$.
\end{definition}

We call $K_0(\St_B^{\aff})$ the \emph{Grothendieck ring of $B$-stacks}, the ring structure being given by fibre product as in \Cref{product-K-ring}. By results of Kresch \cite[Section 4]{kreschcycle}, this ring is isomorphic to the localisation
\[
K_0(\Var_B)[\BL^{-1},(\BL^n-1)^{-1}\,|\,n\geqslant 1].
\]
If $B=\Spec \bfk$, this is equivalent to the localisation of $K_0(\Var_{\bfk})$ at the classes 
\[
[\GL_n] = \prod_{i=0}^{n-1}\,\left(\BL^{n}-\BL^i\right),\quad n\geqslant 1,
\]
or, equivalently, at the classes of all special algebraic groups, where a smooth algebraic group $G$ is called \emph{special} if every principal $G$-bundle over a $\bfk$-variety is Zariski locally trivial. See Serre \cite[Sec.~4.1]{Serre-GLn} for the original definition and Grothendieck \cite{Grothendieck-torsion} for classification results.

\begin{remark}
It follows from these results and definitions that if $G$ is a special algebraic group over $\bfk$ and $Y$ is a $\bfk$-variety on which $G$ acts, then for the quotient stack $\CX = [Y/G]$ we have an identity
\[
[\CX] = \frac{[Y]}{[G]}\,\in\,K_0(\St_{\bfk}^{\aff}).
\]
This was proved by Ekedahl in \cite[Prop.~1.1]{EKStacks}.
In particular, the motive of a quotient stack does not see the action. For instance, the motivic class of $\mathrm B\BG_{m,\bfk} = [\Spec \bfk/\BG_{m,\bfk}]$ is $1/(\BL-1)$.
We will use this fact without further mention.
\end{remark}

\subsubsection{Equivariant Grothendieck rings}
The material in this subsection will only be needed in \Cref{sec:motivic-series}. For more details (or context), we refer the reader to  \cite{DavisonR}, or \cite{DenefLoeser1}. Recall that an action of a finite group $G$ on a $\bfk$-scheme $B$ is said to be \emph{good} if every point $x \in X$ is contained in a $G$-invariant affine open subset of $B$. Every finite group action on a variety is good, and in this case a quotient $B/G$ exists as a variety.

\begin{definition}\label{Equivariant_K_Group}
Let $G$ be a finite group, $B$ a $\bfk$-variety with a $G$-action. We denote by $\widetilde{K}_0^{G}(\Var_B)$ the abelian group 
generated by isomorphism classes $[X\to B]$ of $G$-equivariant $B$-varieties
with good action, modulo the $G$-equivariant scissor relations.
We define the $G$-\emph{equivariant Grothendieck group} $K_0^{G}(\Var_B)$ by imposing the further relations $[V\to X \to B]= [\BA^r_X]$,
whenever $V\to X$ is a $G$-equivariant vector bundle of rank $r$, with $X\to B$ a $G$-equivariant $B$-variety.
The element $[\BA^r_X]$ in the right hand side is taken with the $G$-action induced by the trivial action on $\BA^r$ and the isomorphism $\BA^r_X=\BA^r\times X$.
\end{definition}

There is a natural ring structure on $\widetilde{K}_0^{G}(\Var_B)$ given by taking the diagonal action on $X\times_BY$, for two equivariant $B$-varieties $X\to B$ and $Y\to B$. These rings are related by a $K_0(\Var_{B/G})$-linear map \cite[Lemma 1.5]{DavisonR}
\begin{equation}\label{map:quot1map}
\begin{tikzcd}
\widetilde{K}_0^{G}(\Var_B) \arrow{r}{\pi_G} &  K_0(\Var_{B/G})   
\end{tikzcd}
\end{equation}
defined on generators by taking the orbit space,
\[
\begin{tikzcd}
{[}X\to B{]} \arrow[mapsto]{r} & {[}X/G \to B/G{]}.
\end{tikzcd}
\]

Let $n>0$ be an integer, and let $\mathfrak{S}_n$ be the symmetric group of $n$ elements. By \cite[Lemma~1.6]{DavisonR}, there exists an `$n$-th power' map
\begin{equation}\label{powermap}
\begin{tikzcd}
K_0(\Var_B)\arrow{r}{(\,\cdot\,)^{\otimes n}} & \widetilde{K}_0^{\,\mathfrak S_n}(\Var_{B^n})
\end{tikzcd}
\end{equation}
where the $n$-fold product $B^n = B\times \cdots\times B$ carries the natural $\mathfrak{S}_n$-action.

\subsubsection{The language of power structures}\label{subsec:power-str}
We follow \cite{GLMps}. Roughly speaking, a \emph{power structure} on a ring $R$ is a rule allowing one to raise a power series $A(t) \in R\llbracket t \rrbracket$ starting with $1 \in R$ to an element $r \in R$, in such a way that all of the familiar properties of `raising to a power' are satisfied.

\begin{definition}[{\cite{GLMps}}]\label{def:power_structure}
A \emph{power structure} on a ring $R$ is a map 
\begin{align*}
(1+tR\llbracket t\rrbracket)\times R&\to 1+tR\llbracket t\rrbracket\\
(A(t),m)&\mapsto A(t)^m
\end{align*}
satisfying the following conditions: 
\begin{enumerate}
    \item $A(t)^0=1$,
    \item $A(t)^1=A(t)$, 
    \item $(A(t) B(t))^m=A(t)^m B(t)^m$, 
    \item $A(t)^{m+m'}=A(t)^m A(t)^{m'}$, 
    \item $A(t)^{mm'}=(A(t)^m)^{m'}$, 
    \item $(1+t)^m=1+mt+\mathrm{O}(t^2)$, 
    \item $A(t)^m\big{|}_{t\to t^e}=A(t^e)^m$.
\end{enumerate}  
\end{definition}

Fix a $\bfk$-variety $X$ and a power series $A(t)=1+\sum_{n>0}A_nt^n \in K_0(\Var_\bfk)\llbracket t\rrbracket$. Define
\begin{equation}\label{eqn:power_formula}
A(t)^{[X]}=1+\sum_{n > 0}\sum_{\alpha\vdash n}\pi_{\mathfrak{S}_\alpha}\Biggl(\Biggl[\prod_i X^{\alpha_i}\setminus \Delta\Biggr]\cdot \prod_i A_i^{\otimes \alpha_i}\Biggr)t^{n}.
\end{equation}
Here $\Delta\subset \prod_i X^{\alpha_i}$  is the `big diagonal', namely the locus in the product where at least two entries are equal; the group $\mathfrak{S}_\alpha$ is the automorphism group of the partition $\alpha = (1^{\alpha_1}\cdots n^{\alpha_n})$, namely $\mathfrak{S}_\alpha = \prod_i\FS_{\alpha_i}$. The product in big round brackets 
is a $\mathfrak{S}_\alpha$-equivariant motive, thanks to the power map \eqref{powermap}; finally, $\pi_{\mathfrak{S}_\alpha}$ is the quotient map \eqref{map:quot1map}.
Gusein-Zade, Luengo and Melle-Hern{\'a}ndez  proved in \cite[Thm.~2]{GLMps} that there is a unique power structure 
\[
(A(t),m) \mapsto A(t)^{m}
\]
on $K_0(\Var_{\bfk})$ such that \Cref{eqn:power_formula} holds whenever $m$ along with all the coefficients of $A(t)$ are \emph{effective} (i.e.~classes of actual $\bfk$-varieties). 
There is an important `geometric interpretation' of the power structure, found by Gusein-Zade, Luengo and Melle-Hern{\'a}ndez. It goes as follows. Suppose $X$ is a $\bfk$-variety, and $(A_n)_{n>0}$ is a sequence of $\bfk$-varieties. Form the generating series $A(t) = 1+\sum_{n>0}[A_n]t^n$. Then the $n$-th coefficient of $A(t)^{[X]}$ is the class of the $\bfk$-variety
\[
Y_n = \coprod_{\alpha\vdash n} \left(\prod_i X^{\alpha_i} \setminus \Delta\right) \times \left(\prod_i A_i^{\alpha_i}\right) \Bigg{/}\mathfrak{S}_\alpha
\]
where $\mathfrak{S}_\alpha$ acts on each factor separately.

Attached to a power structure is the \emph{plethystic exponential}. It is the operator sending
\[
\begin{tikzcd}
f=\displaystyle\sum_{n>0}\,A_nt^n \arrow[mapsto]{r} & \Exp(f) = \displaystyle\prod_{n>0}\left(1-t^n\right)^{-A_n}.
\end{tikzcd}
\]
For instance, if $X$ is a $\bfk$-variety, one has
\[
\Exp([X]t) = (1-t)^{-[X]} = \zeta_X(t),
\]
where
\[
\zeta_X(t) = \sum_{n\geqslant 0}\,[\Sym^n(X)]t^n
\]
is the Kapranov zeta function \cite{Kapranov-zeta}.

\begin{remark}\label{rmk:ps-extends}
The power structure we described above has a unique extension to $K_0(\St^{\aff}_{\bfk})$, see \cite[Remark 3.7]{BenSven2}. Moreover, 
the geometric interpretation of the power structure on $K_0(\Var_{\bfk})$ extends to $K_0(\St^{\aff}_{\bfk})$, see \cite[Lemma 5]{BM15}. We shall use this fact in \Cref{thm:gen-fct-coh}.
\end{remark}

%%%%%%%%%%%%%%%%%%%%%%%%%%%%%%%%%%%%%%%%%%%%%%%%%%%%%%%%%%%%%%%%%%%%%%%%%%%%%%%%%%%%%%%%%%%%%
\section{The stack of coherent sheaves}\label{sec:coh(X/B)}
\subsection{\texorpdfstring{The stack $\CCoh(X/B)$ for a quasiprojective family}{}}\label{sec:general-stack}
Let $B$ be a locally noetherian scheme, $f\colon X\to B$ a quasiprojective morphism. For any base change $T \to B$ we shall assume $T$ to be locally noetherian.\footnote{The assumptions on $f$ and $T$ may be relaxed, at the expense of defining a family of coherent sheaves as a finitely presented $T$-flat quasicoherent sheaf (with proper support)
\cite[\href{https://stacks.math.columbia.edu/tag/08KA}{Tag 08KA}]{stacks-project}.}
\begin{notation}\label{notation:basechange}
Given a $B$-scheme $T \to B$, we use the shorthand notation 
\[
X_T = X \times_BT.
\] 
Also, if $\cF \in \Coh(X_T)$ and $h\colon T' \to T$ is a $B$-morphism, we denote by $\cF_{T'}$ the pullback of $\cF$ along the base change map $h_X \colon X_{T'} \to X_T$. It is a coherent sheaf by our assumptions.
\end{notation}

\begin{definition}\label{def:Coh(X/B)}
Define 
\[
\CCoh(X/B)
\]
to be the category whose objects are pairs
\[
\left(g\colon T \to B,\cF\right),
\]
where $\cF \in \Coh(X_T)$ is a family of coherent sheaves as in \Cref{def:family-coh-sheaves}, and whose morphisms
\[
\begin{tikzcd}
\left(g'\colon T' \to B,\cF'\right) \arrow{r} & \left(g\colon T \to B,\cF\right)
\end{tikzcd}
\]
between two objects are pairs
\[
\left(h\colon T' \to T, \phi \right)
\]
where $h$ is a $B$-morphism of schemes, and $\phi \colon \cF_{T'} \simto \cF'$ is an isomorphism in $\Coh(X_{T'})$. Here $\cF_{T'} = h_X^\ast \cF$ according to \Cref{notation:basechange}.
\end{definition}

It is proven in \cite[\href{https://stacks.math.columbia.edu/tag/08KA}{Tag 08KA}]{stacks-project} that the natural functor
\[
\begin{tikzcd}
p\colon \CCoh(X/B) \arrow{r} & \Sch_{B}
\end{tikzcd}
\]
sending $(T \to B, \cF)\mapsto (T \to B)$ defines a $B$-stack (cf.~\Cref{def:stack}). The root of this fact can be found in Grothendieck's theorem \cite[Sec.~B]{Grothendieck_fpqc} stating that quasicoherent sheaves (and their morphisms) satisfy \emph{fpqc descent}. See also \cite[Thm.~4.23]{MR2223406} or \cite[Sec.~4.3]{Olsson_book} for other references. 

One has the following general result (but see also \Cref{rmk:generalisations}).

\begin{theorem}
[{\cite[\href{https://stacks.math.columbia.edu/tag/08WB}{Tag 08WB}]{stacks-project}}] 
\label{thm:stack-coh}
The $B$-stack $\CCoh(X/B)$ is algebraic.
\end{theorem}

We will not give a proof of this general result. We will, however, provide a sketch of its proof under some additional assumptions, following \cite[Thm.~4.6.2.1]{LMB}. In this situation, we shall see an explicit atlas (cf.~\Cref{def:alg-dm-atlas}) for the stack $\CCoh(X/B)$, in terms of Quot schemes.

\begin{theorem}\label{thm:coh-is-algebraic}
Let $B$ be a noetherian scheme, $f \colon X \to B$ a projective morphism such that $f_\ast \OO_X = \OO_B$ holds universally. Then the $B$-stack $\CCoh(X/B)$ is algebraic and has affine diagonal.
\end{theorem}

\begin{proof}
We split the proof in several steps. We grant the fact (mentioned before the statement of \Cref{thm:stack-coh}) that $\CCoh(X/B)$ satisfies the stack axioms. Let us write $\curly{C}= \CCoh(X/B)$ to ease the notation.

\smallbreak
\noindent
\textbf{Step 1: the diagonal.} Here, we verify that the diagonal
\[
\begin{tikzcd}
\Delta \colon \curly{C}\arrow{r} & \curly{C}\times_B \curly{C}
\end{tikzcd}
\]
is representable and affine. Our goal is to check that if $T \to B$ is a $B$-scheme and $\cF,\cG$ are $T$-flat coherent sheaves on $X_T = X \times_BT$, then the set-valued functor
\[
\begin{tikzcd}[row sep=tiny]
  \Sch_{T}^{\opp} \arrow{r}{\vartheta^T_{\cF,\cG}} & \Sets \\
  (U \to T)\arrow[mapsto]{r} & \Iso_{X_U}(\cF_U,\cG_U)
\end{tikzcd}
\]
is represented by an affine $T$-scheme. 

By noetherian descent, we may assume that $T$ is noetherian, and, in fact, that $T=B$. This implies that $X$ is noetherian as well, since $X \to B$ is projective.

\smallbreak
\noindent
\textbf{Step 1.1: the functor $\mathsf{Hom}(\cF,\cG)$.} We apply \cite[Cor.~7.7.8, Rmq.~7.7.9]{EGA32} (see also \cite[Thm.~5.8]{Nit} for another reference), which implies the following: if $B$ noetherian, $X \to B$ is projective, and $\cF,\cG$ are coherent sheaves on $X$, with $\cG$ flat over $B$, then the functor 
\[
\begin{tikzcd}[row sep=tiny]
  \Sch_{B}^{\opp} \arrow{r}{\mathsf{Hom}(\cF,\cG)} & \Sets \\
  (U \to B)\arrow[mapsto]{r} & \Hom_{X_U}(\cF_U,\cG_U)
\end{tikzcd}
\]
is represented by a linear $B$-scheme, i.e.~an affine finite type $B$-scheme of the form 
\[
\begin{tikzcd}
\mathbf V_{\cF,\cG} = \Spec_{\OO_B} \Sym \CV \arrow{r} & B,
\end{tikzcd}
\]
where $\CV \in \Coh(B)$. Since $\cF$ and $\cG$ are both $B$-flat, the conclusion on representability holds also for $\mathsf{Hom}(\cG,\cF)$, $\mathsf{Hom}(\cF,\cF)$ and $\mathsf{Hom}(\cG,\cG)$.

\smallbreak
\noindent
\textbf{Step 1.2: back to $\vartheta^B_{\cF,\cG}$.}
The closed subfunctor $\CY \subset \mathsf{Hom}(\cF,\cG) \times \mathsf{Hom}(\cG,\cF)$ sending $U \to B$ to the set
\[
\CY(U) = \Set{\left(\alpha \colon \cF_U \to \cG_U, \beta\colon \cG_U \to \cF_U\right)|\alpha\circ \beta = \id_{\cG_U},\, \beta \circ \alpha = \id_{\cF_U}}
\]
is representable (by the last sentence of the previous step) by a closed subscheme $Y$ of the affine $B$-scheme $\mathbf V_{\cF,\cG} \times_B \mathbf V_{\cG,\cF}$. It agrees with the functor $\vartheta^B_{\cF,\cG}$. More precisely, the projection $(\alpha,\beta) \mapsto \alpha$ defines a finite type, formally \'etale monomorphism (i.e.~an open immersion) $Y \to \mathbf V_{\cF,\cG}$ with image precisely $\vartheta^B_{\cF,\cG}$. 

Therefore the diagonal is representable and affine, since $Y\to B$ is affine.\footnote{See also \cite[\href{https://stacks.math.columbia.edu/tag/08K9}{Tag 08K9}]{stacks-project} for a generalisation.}

\smallbreak
\noindent
\textbf{Step 2: The atlas.} 
Fix an $f$-very ample invertible sheaf $\OO_X(1)$ on our projective family $f \colon X \to B$. Given $\cF \in \Coh(X)$, Serre's theorem ensures that, at least Zariski locally on the base, $\cF$ receives a surjection
\[
\begin{tikzcd}
\OO_X(-m)^{\oplus d}\arrow[two heads]{r} & \cF,
\end{tikzcd}
\]
for some $d>0$ and some $m \in \BZ$. For any such pair $(m,d)$, construct the subfunctor 
\[
\mathsf Q_{m,d} \subset \mathsf{Quot}_{X/B}(\OO_X(-m)^{\oplus d})
\]
declaring $\mathsf Q_{m,d}(U \to B)$ to be the set of isomorphism classes of $U$-flat quotients 
\begin{equation}\label{eqn:surj-m,d}
\begin{tikzcd}
\OO_{X_U}(-m)^{\oplus d} \arrow[two heads]{r}{\alpha} & \cF \,\,\,\mbox{ on } X_U = X \times_BU  
\end{tikzcd}
\end{equation}
such that \begin{itemize}
    \item [\mylabel{cond1}{(i)}] $R^{>0}f_{U\ast}\cF(m) = 0$, and
    \item [\mylabel{cond2}{(ii)}] the induced map $f_{U\ast}\alpha \colon \OO_U^{\oplus d} \to f_{U\ast}\cF(m)$ is an isomorphism.
\end{itemize}
Here $f_U \colon X_U \to U$ is the base change of $f\colon X \to B$ along $U \to B$. The assumption that $f_\ast \OO_X = \OO_B$ holds universally allows us to make sense of Condition \ref{cond2} and to identify surjections as in \eqref{eqn:surj-m,d} with morphisms $\OO_U^{\oplus d} \to f_{U\ast}\cF(m)$ such that their adjoint map $\OO_{X_U}^{\oplus d} \to f_U^\ast f_{U\ast}\cF(m) \to \cF(m)$ is surjective.

To confirm that $\mathsf Q_{m,d}$ is, indeed, a subfunctor, consider a cartesian diagram
\[
\begin{tikzcd}
X_{U'}\arrow{d}{f_{U'}}\arrow{r}{h_X} & X_U\arrow{d}{f_U}\arrow{r}{} & X\arrow{d}{f} \\
U'\arrow{r}{h} & U\arrow{r} & B
\end{tikzcd}
\]
and fix $[\alpha\colon \OO_{X_U}(-m)^{\oplus d} \onto \cF] \in \mathsf Q_{m,d}(U \to B)$. The `pulllback' map sends this object to $[h_X^\ast \alpha \colon \OO_{X_{U'}}(-m)^{\oplus d} \onto h_X^\ast\cF]$. It is clear that Condition \ref{cond1} continues to hold after base change. The adjoint map to $h_X^\ast \alpha$ is
\[
\begin{tikzcd}
f_{U'\ast}h_X^\ast \alpha = h^\ast f_{U\ast}\alpha\colon \OO_{U'}^{\oplus d} \arrow{r} & h^\ast f_{U\ast}\cF(m) = f_{U'\ast}h_X^\ast\cF(m),
\end{tikzcd}
\]
where we have used that $f_{U\ast}$ commutes with base change thanks to Condition \ref{cond1}. It is therefore again an isomorphism, being the pullback of an isomorphism.

Now, $\mathsf Q_{m,d}$ is an open subfunctor (reason: Condition \ref{cond1} is open by the semicontinuity theorem, and Condition \ref{cond2} is open after imposing Condition \ref{cond1}), and hence, by representability of the Quot functor, it is representable by an open subscheme $\mathrm{Q}_{m,d}$ of $\Quot_{X/B}(\OO_X(-m)^{\oplus d})$. Moreover, there is a 1-morphism
\[
\begin{tikzcd}
\tau_{m,d} \colon \mathrm{Q}_{m,d} \arrow{r} & \curly{C}
\end{tikzcd}
\]
sending an arbitrary  $U$-valued point $[\OO_{X_U}(-m)^{\oplus d} \onto \cF] \in \mathrm{Q}_{m,d}(U\to B)$ to the object $\cF \in \curly{C}(U \to B)$. We can thus form the 1-morphism
\[
\begin{tikzcd}
\tau \colon \displaystyle\coprod_{m,d} \mathrm{Q}_{m,d} \arrow{r} & \curly{C}.
\end{tikzcd}
\]
We claim that $\tau$ is smooth and surjective, hence an atlas for $\curly{C}$. Surjectivity of $\tau$ follows from Serre's theorem: for $m\gg 0$ one has that $R^{>0}f_{U\ast}\cF(m)=0$, $f_{U\ast}\cF(m)$ is finite locally free and $f_U^\ast f_{U\ast}\cF(m) \to \cF(m)$ is surjective. Restricting to a trivialising open for $f_{U\ast}\cF(m)$, we see that $\cF$ belongs to $\mathrm{im}(\tau_{m,d})$, where $d = \rk f_{U\ast}\cF(m)$. So we are left with proving smoothness.

Let us consider a $B$-scheme $U$, and pick a $U$-valued point $U \to \curly{C}$, corresponding to $\cF \in \Coh(X_U)$. Let $U(m,d) \subset U$ be the largest open subset over which $R^{>0}f_{U\ast}\cF(m) = 0$, the evaluation map $f_U^\ast f_{U\ast}\cF(m) \to \cF(m)$ is surjective, and $f_{U\ast}\cF(m)$ is locally free of rank $d$ (all three are open conditions). Then there is a factorisation
\[
\begin{tikzcd}
    U \times_{\curly{C}} \mathrm{Q}_{m,d} \arrow[swap]{dr}{q_{m,d}}\arrow{rr}{\mathrm{pr_1}} & & U \\
    & U(m,d)\arrow[hook]{ur} &
\end{tikzcd}
\]
where the map $q_{m,d}$ is a $\GL_d$-torsor, canonically identified with the projection
\[
\begin{tikzcd}
\Iso(\OO_{U(m,d)}^{\oplus d},f_{U(m,d)\ast}\cF_{U(m,d)}) \arrow{r} & U(m,d).
\end{tikzcd}
\]
In particular, each $q_{m,d}$ is smooth and surjective. 
\end{proof}

\begin{remark}\label{rmk:generalisations}
The reader interested in further generalisations of the setup in which one can construct a stack of coherent sheaves can consult the papers \cite{zbMATH05081797,zbMATH06477021} and the references therein, or the Stacks Project \cite[\href{https://stacks.math.columbia.edu/tag/08KA}{Tag 08KA}]{stacks-project}.
\end{remark}

\subsection{The stack of 0-dimensional sheaves}
\label{sec:0-dim}
Let $X$ be a $\bfk$-variety. As recalled in \cite[Sec.~3.1]{Fantechi-Ricolfi-structural}, the  stack $\CCoh(X/\bfk)$, which is algebraic by \Cref{thm:stack-coh}, contains an open substack $\CCoh(X/\bfk)_0$ parametrising $0$-dimensional coherent sheaves. Openness follows immediately from semicontinuity of fibre dimension, and uses properness of the support. The stratification by the length $n=\chi(\CF)$ of the sheaves $\CF$ parametrised by $\CCoh(X/\bfk)_0$ yields open and closed substacks 
\[
\CCoh^n(X) \subset \CCoh(X/\bfk)_0.
\]
We refer to \cite[Sec.~3.1]{Fantechi-Ricolfi-structural} for full details. 
It is proven in \cite[Thm.~3.3]{Fantechi-Ricolfi-structural} that $\CCoh^n(X)$ is a global quotient stack $[A_{X,n}/\GL_n]$ if $X$ is projective, where $A_{X,n}$ is an open subscheme of $\Quot_X(\OO_X^{\oplus n},n)$.

\subsubsection{Sheaves supported on a locally closed subscheme}
\label{sec:support-on-subset}
Let $X$ be a $\bfk$-variety, $Z \into X$ a closed subscheme with ideal sheaf $\mathscr I_Z \subset \OO_X$. Let $Z_k \into X$ be the closed subscheme defined by $\mathscr I_Z^k$. Then the chain of closed immersions $Z=Z_1 \into Z_2 \into Z_3 \into \cdots$ induces, for every $n\geqslant 0$, a chain of closed immersions of algebraic stacks
\[
\begin{tikzcd}
\CCoh^n(Z)\arrow[hook]{r} & \CCoh^n(Z_2)\arrow[hook]{r} & \CCoh^n(Z_3)\arrow[hook]{r} & \cdots
\end{tikzcd}
\]
with the property that 
\[
\begin{tikzcd}
\CCoh^n(Z_n) \arrow{r}{\sim} & \CCoh^n(Z_m)
\end{tikzcd}
\]
is an isomorphism for every $m \geqslant n$. We then define
\[
\CCoh^n(X,Z) \defeq \CCoh^n(Z_n).
\]
The $\bfk$-valued points of $\CCoh^n(X,Z)$ are exactly the length $n$ sheaves on $X$ with set-theoretic support $Z$. Clearly this construction can be performed for any locally closed subscheme $Z \into X$, by fixing an open subset $U \subset X$ containing $Z$ as a closed subscheme. 

Furthermore, the same procedure leads, for every $r>0$, to the definition of a locally closed subscheme of $\Quot_X(\OO_X^{\oplus r},n)$ parametrising quotients supported set-theoretically on $Z$, namely one defines
\[
\Quot_X(\OO_X^{\oplus r},n,Z) \defeq \Quot_{Z_n}(\OO_{Z_n}^{\oplus r},n).
\]
When $r=1$ and $Z$ is reduced to a closed point $p \in X$, this leads to the definition of the punctual Hilbert scheme as
\[
\Hilb^n(X,p) = \Hilb^n(\Spec \OO_{X,p}/\mathfrak m_p^n).
\]
This is naturally a closed subscheme of $\Hilb^n(X)$, as the morphism $\Spec \OO_{X,p}/\mathfrak m_p^n \to X$ is a closed immersion.

%%%%%%%%%%%%%%%%%%%%%%%%%%%%%%%%%%%%%%%%%%%%%%%%%%%%%%%%%%%%%%%%%%%%%%%%%%%%%%%%%%%%%%%%%%%%%
%%%%%%%%%%%%%%%%%%%%%%%%%%%%%%%%%%%%%%%%%%%%%%%%%%%%%%%%%%%%%%%%%%%%%%%%%%%%%%%%%%%%%%%%%%%%%
\section{Coh-to-Chow and Quot-to-Chow morphisms}

\subsection{The symmetric product}\label{sec:sym^n(X)}
Let $S$ be a scheme. Let $X$ be a quasiprojective $S$-scheme, and fix an integer $n \geqslant 0$. Let $\FS_n$ denote the symmetric group on $n$ letters. The GIT quotient 
\[
\Sym^n(X/S) = X^n/\FS_n
\]
is the $n$-th \emph{symmetric product} of $X \to S$, where $X^n = X\times_S\cdots\times_SX$ ($n$ times). It is a quasiprojective $S$-scheme, representing the functor of families of relative effective $0$-cycles on $X\to S$, and is for this reason also called the \emph{Chow scheme} of relative $0$-cycles of degree $n$ on $X\to S$. We refer to Rydh \cite{Rydh1} for this approach.
We have 
\[
\Sym^0(X/S)=S,\qquad \Sym^1(X/S)=X.
\]

The key example for us is the case where $S=\Spec \bfk$ and $X$ is a $\bfk$-variety, in which case we set $\Sym^n(X) = \Sym^n(X/\Spec \bfk)$. The $\bfk$-points of $\Sym^n(X)$ can be represented as finite formal sums
\[
\sum_{j}n_j x_j
\]
of closed points $x_j \in X$, with nonnegative coefficients $n_j \in \BN$, subject to the condition $\sum_j n_j = n$.

\begin{remark}
If $X = \Spec A$ is an affine $\bfk$-variety, then $\Sym^n(X) = \Spec \,(A^{\otimes n})^{\FS_n}$, and the map $X^n \to \Sym^n(X)$ induced by the inclusion $(A^{\otimes n})^{\FS_n} \into A^{\otimes n}$ is a universal good quotient \cite{Mumford-GIT}.
\end{remark}

\begin{remark}
If $X$ is a smooth $\bfk$-variety of dimension $d>0$, the symmetric product $\Sym^n(X)$ is smooth if and only if $d=1$ or $n\leqslant 1$. In general, $\Sym^n(X)$ has finite quotient singularities, which is typically hard to control.
\end{remark}

\begin{lemma}
\label{rmk:Sym(closed-open)}
Let $X$ be a $\bfk$-variety. If $\iota\colon Z \into X$ is a closed (resp.~open, resp.~locally closed) subscheme of $X$, there is a canonical closed (resp.~open, resp.~locally closed) immersion
\[
\begin{tikzcd}
    \Sym^n(\iota) \colon \Sym^n(Z) \arrow[hook]{r} & \Sym^n(X).
\end{tikzcd}
\]
\end{lemma}
\begin{proof}
The morphism $\Sym^n(\iota)$ is obtained from $\mathfrak S_n$-equivariance of the product map $\iota^n\colon Z^n \into X^n$. As for the case of a closed immersion, it is enough to prove the claim in the affine case, since closed immersions are affine-local on the target. In the affine case, for a closed immersion induced by a ring homomorphism $A \onto B$, we have a surjection $A^{\otimes n} \onto B^{\otimes n}$, which stays surjective after taking $\mathfrak S_n$-invariants since, in characteristic 0, taking invariants is an exact functor for all finite groups. 

If $\iota$ is an open immersion, it is easily checked that $\Sym^n(\iota)$ is \'etale and injective, hence an open immersion. 

The locally closed case is a combination of the previous two.
\end{proof}

\subsubsection{\texorpdfstring{Stratification by partitions for $\Sym$}{}}
There is a (locally closed) stratification of the symmetric product 
\begin{equation}\label{eqn:sm-strata-alpha}
\Sym^n(X) = \coprod_{\alpha\vdash n}\Sym^n_\alpha(X)
\end{equation}
indexed by partitions $\alpha$ of the integer $n$. We use the notation $\alpha = (1^{\alpha_1}\cdots i^{\alpha_i}\cdots n^{\alpha_n})$ for a partition, meaning that $\alpha$ consists of $\alpha_i$ parts of size $i$ for every $i=1,\ldots,n$. In particular, $n = \sum_i i \alpha_i$. For instance, the `deep stratum' $\Sym^n_{(n^1)}(X)$ is a copy of $X$, embedded in $\Sym^n(X)$ as $p \mapsto n\cdot p$. This stratum is closed and is sometimes called the `small diagonal'.

\subsection{The Coh-to-Chow morphism and its fibres}\label{sec:coh-to-sym}
Let $X \to S$ be a quasiprojective morphism. The (relative) \emph{Coh-to-Chow morphism}, also called \emph{support map}, is the $S$-morphism
\begin{equation}\label{support-map}
\begin{tikzcd}
\supp^n_{X/S}\colon \CCoh^n(X/S) \arrow{r} & \Sym^n(X/S)
\end{tikzcd} 
\end{equation}
taking (in each fibre) a $0$-dimensional coherent sheaf $\CF$ to its support, namely the 0-cycle
\[
\cycle(\CF) = \sum_{x \in X}\length_{\OO_{X,x}}\CF_x\cdot x.
\]
The general construction of \eqref{support-map}, under minimal assumptions, is due to Rydh \cite{Rydh1}. We review Rydh's construction in \Cref{sec:supp-map}. See also \cite[Sec.~4.2]{Fantechi-Ricolfi-structural} for a construction of the `naive' Coh-to-Chow map (in characteristic $0$), namely
\[
\begin{tikzcd}
\CCoh^n(X)_{\red} \arrow{r} & \Sym^n(X).
\end{tikzcd}
\]

\begin{remark}\label{rmk:coh-open}
If $X$ is a $\bfk$-variety and $\iota\colon U \into X$ is an open immersion, the commutative diagram
\[
\begin{tikzcd}[row sep=large,column sep=large]
\CCoh^n(U)\arrow[swap]{d}{\supp_U^n}\arrow[hook]{r} & \CCoh^n(X)\arrow{d}{\supp_X^n} \\
\Sym^n(U)\arrow[hook]{r}{\Sym^n(\iota)} & \Sym^n(X)
\end{tikzcd}
\]
is 2-cartesian.
\end{remark}

By \Cref{rmk:Sym(closed-open)}, any locally closed immersion $Z \into X$ yields a locally closed substack
\[
\begin{tikzcd}
\CCoh^n(X)_Z = \CCoh^n(X) \times_{\Sym^n(X)} \Sym^n(Z) \arrow[hook]{r} & \CCoh^n(X).
\end{tikzcd}
\]
We refer to $\CCoh^n(X)_Z$ as the stack of length $n$ sheaves on $X$ \emph{supported on} $Z$. If $U$ is open in $X$, we have $\CCoh^n(X)_U = \CCoh^n(U)$ by \Cref{rmk:coh-open} (see also \cite[Sec.~2.3.1]{Fantechi-Ricolfi-structural}).

An important special case is when $Z \into X$ is the inclusion of a closed point $p \in X$. In this case, we obtain the closed substack
\begin{equation}\label{def:punctual-stack}
\begin{tikzcd}
\CCoh^n(X)_p \arrow[hook]{r} &  \CCoh^n(X)
\end{tikzcd}
\end{equation}
parametrising sheaves entirely supported at $p$. We shall observe that the motivic class of $\CCoh^n(X)$ is, for a smooth variety $X$ of dimension $d$, entirely determined by the classes of $\CCoh^i(\BA^d)_0$ for $i \leqslant n$, and by the motivic class of $X$.

\begin{remark}
\label{rmk:supp-Z}
Let $Z \subset X$ be a locally closed subscheme. It is immediate to check that, at the level of geometric points, there is no difference between the closed substacks $\CCoh^n(X)_Z$ and $\CCoh^n(X,Z)$. This implies that there is an isomorphism at the level of reduced structures, namely
\[
\begin{tikzcd}
\CCoh^n(X,Z)_{\red} \arrow{r}{\sim} & (\CCoh^n(X)_Z)_{\red}.
\end{tikzcd}
\] 
When $Z$ is open, there is no need to pass to the reductions, as both stacks agree with $\CCoh^n(U)$.
\end{remark}

\subsubsection{\texorpdfstring{Stratification by partitions for $\CCoh$}{}}
The stratification \eqref{eqn:sm-strata-alpha}, combined with \Cref{rmk:Sym(closed-open)}, allows us to define, for each partition $\alpha\vdash n$, the locally closed substack 
\[
\begin{tikzcd}
\CCoh^n_\alpha(X) = \CCoh^n(X) \times_{\Sym^n(X)}\Sym^n_\alpha(X) \arrow[hook]{r} & \CCoh^n(X),
\end{tikzcd}
\]
parametrising sheaves whose support is distributed according to $\alpha$. Clearly, for every closed point $p \in X$, the `punctual locus' $\CCoh^n(X)_p$ sits inside the `deepest stratum'
\[
\begin{tikzcd}
\CCoh^n_{(n^1)}(X) \arrow[hook]{r} & \CCoh^n(X),
\end{tikzcd}
\]
the preimage of the small diagonal $X \into \Sym^n(X)$, as a closed substack. We recall the following structural property of the stratification by partitions.

\begin{theorem}[{\cite[Thm.~5.8]{Fantechi-Ricolfi-structural}}]\label{prop:zar-triviality}
Let $X$ be a smooth variety of dimension $d$.
\begin{itemize}
\item [\mylabel{punctual-1}{(1)}] If $X=\BA^d$, the projection
\[
\begin{tikzcd}
\supp^n_{\BA^d} \colon \CCoh^n_{(n^1)}(\BA^d) \arrow{r} & \BA^d
\end{tikzcd}
\]
is a \emph{trivial} fibration with fibre $\CCoh^n(\BA^d)_0$, where $0 \in \BA^d$ denotes the origin.
\item [\mylabel{punctual-2}{(2)}] The projection 
\[
\begin{tikzcd}
\supp^n_{X} \colon \CCoh^n_{(n^1)}(X) \arrow{r} & X
\end{tikzcd}
\]
is \emph{Zariski locally trivial} with fibre $\CCoh^n(\BA^d)_0$.
\item [\mylabel{punctual-3}{(3)}] If $\alpha = (1^{\alpha_1}\cdots i^{\alpha_i}\cdots n^{\alpha_n})$ is a partition of $n$, then
\[
\begin{tikzcd}
\supp^n_X \colon \CCoh^n_\alpha(X) \arrow{r} & \Sym^n_\alpha (X)
\end{tikzcd}
\]
is \emph{\'etale locally trivial} with fibre $\prod_i \CCoh^i(\BA^d)_0^{\alpha_i}$.
\end{itemize}
\end{theorem}

We shall see in \Cref{app:et-vs-zar} that the \'etale locally trivial fibration of \ref{punctual-3} fails to be \emph{Zariski} locally trivial even in very simple situations.

\begin{comment}
We first prove that $t$ is \'etale. Consider the fibre diagram
\[
\begin{tikzcd}[column sep=large,row sep=large]
U\MySymb{dr} \arrow{r}{h}\arrow{d}{a} & \Quot_{\BA^d}(\OO^{\oplus n},n)_{(n^1)} \arrow{d}{b} \\
\BA^d \times \CCoh^n(\BA^d)_0\arrow{r}{t} & \CCoh^n_{(n^1)}(\BA^d)
\end{tikzcd}
\]
where $b$ is the Quot-to-Chow morphism. Then $a$ is surjective and $h$ is \'etale (in fact, an isomorphism, since $U = \BA^d \times \Quot_{\BA^d}(\OO^{\oplus n},n)_0$). Therefore $t$ is \'etale by \cite[\href{https://stacks.math.columbia.edu/tag/0CIK}{Tag 0CIK}]{stacks-project}.

It is clear that $t$ induces equivalences at the level of geometric points, so it remains to prove it induces isomorphisms at the level of automorphism groups.
\end{comment}

%%%%%%%%%%%%%%%%%%%%%%%%%%%%%%%%%%%%%%%%%%%%%%%%%%%%%%%%%%%%%%%%%%%%%%%%%%%%%%%%%%%%%%%%%%%%%
\subsection{The Quot-to-Chow morphism}
As outlined in the introduction, any moduli space involving a flat family of $0$-dimensional sheaves of length $n$ with proper support admits a natural forgetful morphism into $\CCoh^n(X)$. This is the case, for instance, for the Quot scheme of points.

\begin{lemma}[{\cite[Lemma 3.1]{Fantechi-Ricolfi-structural}}]\label{lemma:quot-to-coh}
Let $X$ be a $\bfk$-variety. Fix $\CE \in \Coh(X)$. For every $n \in \BZ_{\geqslant 0}$ there is a morphism 
\begin{equation}\label{rho-morphism}
\begin{tikzcd}
\rho_{\CE,n} \colon \Quot_X(\CE,n) \arrow{r} & \CCoh^n(X)
\end{tikzcd}
\end{equation}
sending a point $[\CE \onto \CF]$ of the Quot scheme to the point $[\CF]$.
\end{lemma}

Note that when $\CE = \OO_X^{\oplus n}$ the morphism \eqref{rho-morphism} is surjective. 

The previous lemma immediately yields the definition of the Quot-to-Chow morphism.

\begin{definition}\label{def:Quot-to-Chow}
Let $X$ be a $\bfk$-variety, $\CE \in \Coh(X)$ a coherent sheaf, $n \in \BZ_{\geqslant 0}$ an integer. The composition 
\begin{equation}\label{quot-to-chow-map}
\begin{tikzcd}
\mathsf{qc}_{\CE,n}\colon \Quot_X(\CE,n) \arrow{r}{\rho_{\CE,n}} & \CCoh^n(X) \arrow{rr}{\supp^n_X} && \Sym^n(X)
\end{tikzcd}    
\end{equation}
is called the \emph{Quot-to-Chow morphism}. When $\CE=\OO_X$, this is the \emph{Hilbert--Chow morphism} (see also \cite[Ex.~4.3.6]{modulisheaves} and \cite[Sec.~2.2]{Bertin-Hilb})
\[
\begin{tikzcd}
\mathsf{hc}_X^n \colon \Hilb^n(X) \arrow{r} & \Sym^n(X),
\end{tikzcd}  
\] 
sending a closed subscheme $Z \into X$ to its underlying $0$-cycle 
\[
\cycle(\OO_Z)=\sum_{z \in Z}\length_{\OO_{X,z}} \OO_{Z,z}\cdot z.
\]
\end{definition}

%%%%%%%%%%%%%%%%%%%%%%%%%%%%%%%%%%%%%%%%%%%%%%%%%%%%%%%%%%%%%%%%%%%%%%%%%%%%%%%%%%%%%%%%%%%%%
\subsection{Motivic decompositions} 

At this point, we do not know yet whether the stack of coherent sheaves admits a motivic class. The next result clarifies the situation. In fact, the next lemma is a special case of \cite[\href{https://stacks.math.columbia.edu/tag/08K9}{Tag 08K9}]{stacks-project}.

\begin{lemma}
\label{lemma:affine_stab}
Let $X$ be a $\bfk$-variety, $Z \into X$ a closed subscheme. Then the substacks $\CCoh^n(X)_Z \into \CCoh^n(X)$ and $\CCoh^n(X,Z) \into \CCoh^n(X)$ have affine geometric stabilisers.
\end{lemma}

\begin{proof}
In the projective case, this follows immediately from Step 1 of the proof of \Cref{thm:coh-is-algebraic} and by \Cref{lemma:iso-aut}. The general quasiprojective case follows from the projective case by base change from any projective compactification of $X$.
\end{proof}

\begin{remark}
\label{rmk:etale-independence}
By \Cref{lemma:affine_stab}, for any $\bfk$-variety $X$, any closed subscheme $Z \into X$ and any integer $n \geqslant 0$ there are well-defined motivic classes
\[
[\CCoh^n(X)],\,\,[\CCoh^n(X)_Z],\,\,[\CCoh^n(X,Z)] \in K_0(\St_{\bfk}^{\aff}),
\]
and by \Cref{rmk:supp-Z}, there are identities
\begin{equation}
\label{eqn:mot-Z}
[\CCoh^n(X,Z)] = [\CCoh^n(X)_Z]
\end{equation}
in $K_0(\St^{\aff}_{\bfk})$.
\end{remark}

The scissor relations yield the following motivic decomposition.

\begin{prop}\label{prop:coh-dec}
Let $X$ be a $\bfk$-variety. Let $\iota\colon Z\into X$ be a closed subscheme, $U = X \setminus Z$ its complement. Then we have an equality in the Grothendieck ring of stacks 
\[
[\CCoh^n(X)]=\sum_{0\leqslant m\leqslant n} \,[\CCoh^m(U)] [\CCoh^{n-m}(X)_Z].
\]
\end{prop}

\begin{proof}
Let $i\colon U \into X$ denote the inclusion. Fix $m \in \set{0,1,\ldots,n}$. We first construct a morphism
\[
\begin{tikzcd}
\CCoh^m(U)_{\red} \times \CCoh^{n-m}(X,Z)_{\red} \arrow{r}{j_m} & \CCoh^{n}(X).   
\end{tikzcd}
\]
If $T$ is a reduced $\bfk$-scheme, denote by $\iota_T$ and $i_T$ the closed immersion $Z \times T \into X \times T$ and the open immersion $U \times T \into X \times T$, respectively. Given $T$-families $\cF \in \CCoh^m(U)(T)$ and $\cG \in \CCoh^{n-m}(X,Z)(T)$, one has that $i_{T\ast}\cF \oplus \iota_{T\ast}\cG$ is still $T$-flat with proper support over $T$ (cf.~\cite[Prop.~2.9]{Fantechi-Ricolfi-structural}). It thus defines a morphism $T \to \CCoh^{n}(X)$, so that $j_m$ is constructed.

Next, if $\CE \in \CCoh^{n}(X)(\bfk)$, let $\CE_Z$ be the maximal subsheaf of $\CE$ such that $\CE_Z \to \iota_\ast \iota^\ast \CE_Z$ is an isomorphism. Since $\CE_U = \CE/\CE_Z$ is entirely supported on $U$, we have $\CE = \CE_Z \oplus \CE_U$, thus the coproduct of the maps $j_0,j_1,\ldots,j_n$ defines a geometric bijection (cf.~\Cref{def:stratification})
\begin{equation}\label{strat-coh}
\begin{tikzcd}
\displaystyle\coprod_{0\leqslant m\leqslant n}\CCoh^m(U)_{\red} \times \CCoh^{n-m}(X,Z)_{\red} \arrow{r} & \CCoh^{n}(X)
\end{tikzcd}
\end{equation}
and by \Cref{eqn:mot-Z} we get the desired motivic identity.
\end{proof}

\begin{notation}\label{notation:quot-with-support}
Let $X$ be a $\bfk$-variety. Fix a coherent sheaf $\CE$ on $X$. Let $Z \into X$ be a locally closed subscheme. Set
\[
\Quot_X(\CE,n)_Z = \mathsf{qc}_{\CE,n}^{-1}(\Sym^n(Z)),
\]
where $\mathsf{qc}_{\CE,n}$ is the Quot-to-Chow morphism \eqref{quot-to-chow-map}. If $Z$ is a closed point $p \in X$, then the closed subscheme
\[
\begin{tikzcd}
\Quot_X(\CE,n)_p \arrow[hook]{r} & \Quot_X(\CE,n)
\end{tikzcd}
\]
is the preimage of the closed point $n\cdot p \in \Sym^n_{(n^1)}(X) \subset \Sym^n(X)$ on the small diagonal.
\end{notation}

\begin{remark}
If $U \into X$ is open, then
\[
\Quot_X(\CE,n)_U = \mathsf{qc}_{\CE,n}^{-1}(\Sym^n(U)) = \Quot_U(\CE|_U,n).
\]    
\end{remark}

The following decomposition is a consequence of \Cref{prop:coh-dec}.

\begin{corollary}\label{cor:quot-dec}
Let $X$ be a $\bfk$-variety. Let $Z\into X$ be a closed subscheme, $U = X\setminus Z$ its complement. Then, for every coherent sheaf $\CE$ over $X$ and every integer $n \geqslant 0$, we have an equality in the Grothendieck ring of $\bfk$-varieties
\[
[\Quot_X(\CE,n)] = \sum_{0\leqslant m\leqslant n}\, [\Quot_U(\CE|_U,m)] [\Quot_X(\CE,n-m)_Z].
\]
\end{corollary}

\begin{proof}
Note that 
\begin{align*}
\Quot_U(\CE|_U,k) &= \rho_{\CE,k}^{-1}\CCoh^k(X)_U \\
\Quot_X(\CE,k)_Z &= \rho_{\CE,k}^{-1}\CCoh^k(X)_Z
\end{align*}
where $\rho_{\CE,k}$ is the map \eqref{rho-morphism}.
It is therefore enough to pullback the geometric bijection \eqref{strat-coh} of \Cref{prop:coh-dec} along $\rho_{\CE,n}\colon \Quot_X(\CE,n) \to \CCoh^n(X)$.
\end{proof}

The proof of \Cref{thm:motivic-dec-intro} is therefore complete.

%%%%%%%%%%%%%%%%%%%%%%%%%%%%%%%%%%%%%%%%%%%%%%%%%%%%%%%%%%%%%%%%%%%%%%%%%%%%%%%%%%%%%%%%%%%%%
%%%%%%%%%%%%%%%%%%%%%%%%%%%%%%%%%%%%%%%%%%%%%%%%%%%%%%%%%%%%%%%%%%%%%%%%%%%%%%%%%%%%%%%%%%%%%
\section{The punctual stack of sheaves}\label{sec:punctual-stacks}
In this section we explore in greater detail the \emph{punctual stack} $\CCoh^n(X)_p$ attached to a closed point $p$ on a $\bfk$-variety $X$.

%%%%%%%%%%%%%%%%%%%%%%%%%%%%%%%%%%%%%%%%%%%%%%%%%%%%%%%%%%%%%%%%%%%%%%%%%%%%%%%%%%%%%%%%%%%%%
\subsection{Dependence on formal neighbourhood: the smooth case}\label{sec:formal-smooth}

We start recalling the following key lemma.

\begin{lemma}[{\cite[Lemma 5.4]{Fantechi-Ricolfi-structural}}]\label{lemma:etale-comparison}
Let $\pi\colon X' \to X$ be an \'etale map of $\bfk$-varieties. Then the direct image of coherent sheaves induces an \'etale morphism
\[
\begin{tikzcd}[row sep=tiny]
V_{\inj} \arrow{r}{\pi_\ast} & \CCoh^n(X)\\
{[}\CE{]} \arrow[mapsto]{r} & {[}\pi_\ast \CE{]}
\end{tikzcd}
\]
where $V_{\inj}\subset \CCoh^n(X')$ is the open substack of sheaves $\CE$ on $X'$ such that $\pi$ is injective on $\Supp_{\mathrm{set}}(\CE)\subset X'$.
\end{lemma}

\begin{theorem}
\label{thm:punctual-smooth-point}
Let $X$ be a $\bfk$-variety, $p \in X$ a smooth closed point with $\dim T_pX=d$, and $n \in \BZ_{\geqslant 0}$ an integer. Then there is an isomorphism of algebraic stacks
\[
\begin{tikzcd}
    \CCoh^n(X)_p \arrow{r}{\sim} & \CCoh^n(\BA^d)_0.
\end{tikzcd}
\]
\end{theorem}

\begin{proof}
We can pick \'etale coordinates around $p$, i.e.~we can find a pair $(A,\pi)$ where $A$ is an open neighbourhood $p \in A \subset X$ and $\pi\colon A \to \BA^d$ is an \'etale map, sending $p$ to the origin $0 \in \BA^d$ \cite[\href{https://stacks.math.columbia.edu/tag/054L}{Tag 054L}]{stacks-project}. The neighbourhood $A$ can be shrunk to achieve $A_{\pi(p)} = \pi^{-1}\pi(p) = \set{p}$, see e.g.~\cite[App.~A]{BR18}. Therefore we may assume that $\pi$ is an immersion near $p$. By \Cref{lemma:etale-comparison}, we have an \'etale morphism $\pi_\ast$ appearing in the diagram
\[
\begin{tikzcd}[column sep=large,row sep=large]
V_{\inj} \arrow[hook]{r}{\textrm{open}}\arrow{d}{\pi_\ast} & \CCoh^n(A) \\
\CCoh^n(\BA^d) & 
\end{tikzcd}
\]
and its pullback along the closed immersion $\CCoh^n(\BA^d)_0 \into \CCoh^n(\BA^d)$ is, by base change, an \'etale map
\[
\begin{tikzcd}
\CCoh^n(X)_p \arrow{r}{\pi_\ast} & \CCoh^n(\BA^d)_0.
\end{tikzcd}
\]
Since it is \'etale (by pullback) and bijective on points (it is injective because $\pi$ is an immersion near $p$ and surjective because $A_{\pi(p)} = \set{p}$), to conclude it is an isomorphism of stacks it is enough by \cite[\href{https://stacks.math.columbia.edu/tag/0DU6}{Tag 0DU6}]{stacks-project} to check it induces an isomorphism on stabiliser groups. But because $\pi$ is an immersion near $p$ we have, for all $[\CE] \in \CCoh^n(X)_p$, canonical isomorphisms
\[
\begin{tikzcd}
\Hom_X(\CE,\CE) \arrow{r}{\sim} & \Hom_X(\pi^\ast \pi_\ast \CE,\CE) \arrow{r}{\sim} & \Hom_{\BA^d}(\pi_\ast \CE,\pi_\ast \CE),
\end{tikzcd}
\]
the first isomorphism being induced by $\pi^\ast \pi_\ast \CE \simto \CE$, the second one being the natural adjunction $(\pi^\ast,\pi_\ast)$. This induces the sought after isomorphism on stabiliser groups (cf.~\Cref{eqn:isom-map})
\[
\begin{tikzcd}
\Aut(\CE) \arrow{r}{\sim} & \Aut(\pi_\ast \CE),
\end{tikzcd}
\]
which concludes the proof.
\end{proof}

At the level of punctual Quot schemes, we obtain the following result. The proof is very similar to the previous one, we include it for the sake of completeness.

\begin{theorem}
\label{thm:quot-punctual}
Let $X$ be a $\bfk$-variety, $p \in X$ a smooth closed point with $\dim T_pX=d$, and $n \in \BZ_{\geqslant 0}$ an integer. Then, for every $r>0$, there is an isomorphism of schemes 
\[
\begin{tikzcd}
\Quot_{X}(\OO_X^{\oplus r},n)_p  \arrow{r}{\sim} & \Quot_{\BA^d}(\OO_{\BA^d}^{\oplus r},n)_0.
\end{tikzcd}
\] 
\end{theorem}

\begin{proof}
Fix \'etale coordinates as in the proof of \Cref{thm:punctual-smooth-point}, i.e.~find a pair $(A,\pi)$ where $A$ is an open neighbourhood $p \in A \subset X$ and $\pi\colon A \to \BA^d$ is an \'etale map, sending $p$ to the origin $0 \in \BA^d$. Further assume, without loss of generality, that $A_{\pi(p)} = \set{p}$, so that $\pi$ is an immersion near $p$. The proof of \cite[Prop.~A.3]{BR18} shows the existence of an \'etale map 
\[
\begin{tikzcd}
U_{\inj} \arrow{r}{v} & \Quot_{\BA^d}(\OO_{\BA^d}^{\oplus r},n),
\end{tikzcd}
\]
where $U_{\inj}$ is the open subscheme of $\Quot_{A}(\OO_A^{\oplus r},n)$ parametrising isomorphism classes of quotients $\OO_A^{\oplus r} \onto \CF$ such that $\pi$ is injective around $\Supp_{\textrm{set}}(\CF)$. Explicitly, the map $v$ sends a quotient $\theta\colon \OO_A^{\oplus r} \onto \CF$ to 
\[
\begin{tikzcd}
\OO_{\BA^d}^{\oplus r} \arrow{r} & \pi_\ast \pi^\ast \OO_{\BA^d}^{\oplus r} = \pi_\ast \OO_A^{\oplus r} \arrow{r}{\pi_\ast\theta} & \pi_\ast \CF.
\end{tikzcd}
\]
Note that $U_{\inj}$ contains the punctual Quot scheme $\Quot_A(\OO_A^{\oplus r},n)_p = \Quot_X(\OO_X^{\oplus r},n)_p$, and restricts to an \'etale morphism $\Quot_X(\OO_X^{\oplus r},n)_p \to \Quot_{\BA^d}(\OO_{\BA^d}^{\oplus r},n)_0$. This morphism is injective because $\pi$ is an immersion near $p$ and surjective because $A_{\pi(p)} = \set{p}$, hence it is an isomorphism, as required.
\end{proof}

\begin{comment}
\begin{lemma}\label{lemma:punctual}
Let $p \in X$ be a closed point. Let $\mathfrak m_p$ denote the unique closed point of the affine scheme $C = \Spec \widehat{\OO}_{X,p}$.
Fix integers $r > 0$ and $n\geqslant 0$. There is an isomorphism of schemes
\[
\begin{tikzcd}
\sigma\colon \Quot_{C}(\OO_C^{\oplus r},n) \arrow{r}{\sim} & \Quot_{X}(\OO_X^{\oplus r},n)_p.
\end{tikzcd}
\]
\end{lemma}

\begin{proof}
The proof of \cite[Prop.~A.3]{BR18} shows the existence of an \'etale map 
\[
v_{r,n}\colon U \to \Quot_{X}(\OO_X^{\oplus r},n),
\]
where $U$ is the open subscheme of $\Quot_{C}(\OO_C^{\oplus r},n)$ parametrising isomorphism classes of quotients $\OO_C^{\oplus r} \onto \CF$ such that $\pi\colon C \to X$ is injective and \'etale around $\Supp_{\textrm{set}}(\CF)$. But $\Supp_{\textrm{set}}(\CF)$ is reduced to the point $\mathfrak m_p$ and $\pi$ is injective and \'etale around this point, thus $U=\Quot_{C}(\OO_C^{\oplus r},n)$. On points, the map $v_{r,n}$ is defined by sending $\OO_C^{\oplus r} \onto \CF$ to the quotient
\[
\begin{tikzcd}
\OO_X^{\oplus r} \arrow{r} & \pi_\ast \pi^\ast \OO_X^{\oplus r} = \pi_\ast \OO_C^{\oplus r} \arrow{r} & \pi_\ast \CF.
\end{tikzcd}
\]
This is not just \'etale but also \emph{bijective}, hence an isomorphism.
\end{proof}
\end{comment}

The proof of \Cref{main1:punctual-stack} is complete.

\begin{remark}
The isomorphism of Theorems \ref{thm:punctual-smooth-point} and \ref{thm:quot-punctual} are not canonical, for they depend on a choice of \'etale coordinates around $p$.    
\end{remark}

\section{Motivic generating functions}\label{sec:motivic-series}
Let $X$ be a $\bfk$-variety, and fix $\CE \in \Coh(X)$. Consider the generating series
\begin{equation}\label{eqn:motivic-series-quot-coh}
\begin{split}
\mathsf Z_X(t) &= \sum_{n\geqslant 0}\,[\CCoh^n(X)] t^n \,\in\, K_0(\St_{\bfk}^{\aff})\llbracket t\rrbracket\\
\mathsf Q_\CE(t) &= \sum_{n\geqslant 0}\,[\Quot_X(\CE,n)] t^n\,\in\, K_0(\Var_{\bfk})\llbracket t\rrbracket
\end{split}
\end{equation}
first introduced in \eqref{eqn:motivic-series}.

We now exploit the language of power structures (cf.~\Cref{subsec:power-str}) to give a decomposition formula for the series in \eqref{eqn:motivic-series-quot-coh}, under some additional assumptions (cf.~\Cref{thm:gen-fct-coh} and \Cref{thm:Quot-Series}).

\subsection{\texorpdfstring{Decomposition formulas for $\CCoh$ and $\Quot$}{}}
We now give a formula for the generating series \eqref{eqn:motivic-series-quot-coh} in the case where $X$ has $0$-dimensional singular locus. We stress that to make the formula more explicit, one would need to control the motive of the punctual stacks and of the punctual Quot schemes at closed singular points. See the work of Huang--Jiang \cite{Huang-Jiang} for recent progress in this direction.

\begin{definition}\label{def:sing-type}
Two pointed schemes $(Y,y)$ and $(Y',y')$ are \emph{\'etale-equivalent} if there is a ring isomorphism
\[
\begin{tikzcd}
\widehat{\OO}_{Y,y} \arrow{r}{\sim} & \widehat{\OO}_{Y',y'}.
\end{tikzcd}
\]
A \emph{singularity type} is an equivalence class $\sigma = [Y,y]$ of pointed schemes under \'etale-equivalence.
\end{definition}

%See \cite[Lemma 2.1]{Jelisiejew-pathologies} for a proof that smooth equivalence is an equivalence relation.

We start with the stack of 0-dimensional sheaves. Let $\sigma$ be a singularity type. Define the generating function
\[
\mathsf Z_{\sigma}(t) = \sum_{n\geqslant 0} \,[\CCoh^n(Y)_y] t^n \in K_0(\St_{\bfk}^{\aff})\llbracket t \rrbracket,
\]
where $(Y,y)$ is any pointed $\bfk$-variety representing $\sigma$. We set
\[
\sm_d = [\BA^d,0],
\]
the singularity type corresponding to smooth $d$-dimensional $\bfk$-varieties. This yields
\[
\mathsf Z_{\sm_d}(t) = \sum_{n\geqslant 0} \,[\CCoh^n(\BA^d)_0] t^n.
\]
The fact that $\mathsf Z_{\sigma}(t)$ does not depend on the choice of representatives for $\sigma$ is a consequence of \Cref{rmk:etale-independence} and the fact that the stack $\CCoh^n(Y,y)$ only depends on the completion $\widehat{\OO}_{Y,y}$ (see also \Cref{thm:punctual-smooth-point} for the smooth case).

\begin{theorem}
\label{thm:gen-fct-coh}
Let $X$ be a $\bfk$-variety of dimension $d$, with smooth locus $X_{\sm} \subset X$. Let $\sigma_1,\ldots,\sigma_e$ be distinct singularity types, and assume $X$ has $k_i$ points of singularity type $\sigma_i$ for $i=1,\ldots,e$, and no other singularities. Then there are product decompositions
\[
\mathsf Z_X(t) = \mathsf Z_{X_{\sm}}(t) \prod_{1 \leqslant i \leqslant e} \mathsf Z_{\sigma_i}(t)^{k_i} = \mathsf Z_{\sm_d}(t)^{[X_{\sm}]}\prod_{1 \leqslant i \leqslant e} \mathsf Z_{\sigma_i}(t)^{k_i}
\]
in $K_0(\St_{\bfk}^{\aff})\llbracket t \rrbracket$.
\end{theorem}

\begin{proof}
By a slight abuse of notation, write $\CCoh^{a}(X,\sigma)$ for the stack $\CCoh^a(X,p)_{\red}$ for any $p$ such that $[X,p]=\sigma$.
The first identity is then a consequence of the stratification
\[
\CCoh^n(X) = \coprod_{0\leqslant m\leqslant n} \CCoh^m(X_{\sm}) \times \coprod_{\substack{a_1k_1+\cdots+a_ek_e = n-m \\ a_i \geqslant 0}}\prod_{1\leqslant i\leqslant e}\CCoh^{a_i}(X,\sigma_i)^{k_i},
\]
a special case of the stratification \eqref{strat-coh}. The second identity boils down to proving
\[
\mathsf Z_{X_{\sm}}(t) = \mathsf Z_{\sm_d}(t)^{[X_{\sm}]},
\]
which can be confirmed along the same lines of \cite[Thm.~1]{GLMHilb}, expoiting the geometric interpretation of the power structure on $K_0(\St_{\bfk}^{\aff})$, cf.~\Cref{rmk:ps-extends}.
\end{proof}

As for Quot schemes, we have the following situation. Consider a locally free sheaf $\CE$ of rank $r>0$ on a $\bfk$-variety $X$. If $p \in X$ is a closed point, we have an isomorphism
\begin{equation}\label{punctual-iso-quot}
\Quot_X(\CE,n)_p \cong \Quot_X(\OO_X^{\oplus r},n)_p
\end{equation}
and if $X$ is smooth and $d$-dimensional at $p$ we further have
\begin{equation}\label{punctual-quot-A^d}
\Quot_X(\OO_X^{\oplus r},n)_p \cong \Quot_{\BA^d}(\OO_{\BA^d}^{\oplus r},n)_0
\end{equation}
by \Cref{thm:quot-punctual}.
By \Cref{rmk:etale-independence} combined with \eqref{punctual-iso-quot}, the punctual generating function
\[
\sum_{n \geqslant 0}\,[\Quot_X(\CE,n)_p ] t^n = \sum_{n \geqslant 0}\,[\Quot_X(\OO_X^{\oplus r},n)_p ] t^n \in K_0(\Var_{\bfk})\llbracket t \rrbracket
\]
associated to $(X,\CE,p)$ does not depend on $(X,p)$ nor on $\CE$, but only on $r=\rk \CE$ and the singularity type $\sigma = [X,p]$. We may therefore set
\[
\mathsf Q_{r,\sigma}(t) = \sum_{n \geqslant 0}\,[\Quot_X(\OO_X^{\oplus r},n)_p] t^n 
\]
and keep in mind that this series records the motivic classes of all Quot schemes attached to triples $(X,\CE,p)$ where $\CE$ is locally free and $[X,p]=\sigma$. If $\sigma=\sm_d$, by \eqref{punctual-quot-A^d} we obtain (see also \Cref{thm:quot-punctual})
\[
\mathsf Q_{r,\sm_d}(t) = \sum_{n \geqslant 0}\,[\Quot_{\BA^d}(\OO_{\BA^d}^{\oplus r},n)_0 ]t^n.
\]
Once more, we stress that this equals the series
\[
\sum_{n \geqslant 0}\,[\Quot_{X}(\CE,n)_p]t^n
\]
for \emph{any} locally free sheaf $\CE$ (of rank $r$) on \emph{any} smooth $\bfk$-variety $X$ of dimension $d$, and for \emph{any} chosen closed point $p \in X$.

\begin{theorem}\label{thm:Quot-Series}
Let $X$ be a $\bfk$-variety of dimension $d$, with smooth locus $X_{\sm} \subset X$. Let $\sigma_1,\ldots,\sigma_e$ be distinct singularity types, and assume $X$ has $k_i$ points of singularity type $\sigma_i$ for $i=1,\ldots,e$, and no other singularities. Fix a locally free sheaf $\CE$ of rank $r>0$ on $X$. Then there are product decompositions
\[
\mathsf Q_\CE(t) = \mathsf Q_{\CE|_{X_{\sm}}}(t)  \prod_{1\leqslant i\leqslant e} \mathsf Q_{r,\sigma_i}(t)^{k_i} = \mathsf Q_{r,\sm_d}(t)^{[X_{\sm}]} \prod_{1\leqslant i\leqslant e} \mathsf Q_{r,\sigma_i}(t)^{k_i}
\]
in $K_0(\Var_{\bfk})\llbracket t \rrbracket$.
\end{theorem}

\begin{proof}
The first identity is a consequence of the stratification underlying the proof of \Cref{cor:quot-dec}. The second identity boils down to 
\[
\mathsf Q_{\CE|_{X_{\sm}}}(t) = \mathsf Q_{r,\sm_d}(t)^{[X_{\sm}]},
\]
a consequence of the main result of \cite{ricolfi2019motive}, which in turn generalises the $r=1$ case proved by Gusein-Zade, Luengo and Melle-Hern{\'a}ndez \cite{GLMHilb}.
\end{proof}

This concludes the proof of \Cref{main3}.

A version of \Cref{thm:Quot-Series}, where $X$ is a reduced curve and $\CE=\OO_X$, is proved in \cite[Cor.~2.2]{Bejleri_Ranganathan_Vakil_2020}.

\subsection{\texorpdfstring{Explicit formulas for $\CCoh$ and $\Quot$: state of the art}{}}\label{sec:state-of-art}

In low dimension, the series \eqref{eqn:motivic-series-quot-coh} can be computed explicitly. We next review what is known about them.

\begin{prop}
\label{cor:affine_line}
The stack $\CCoh^n(\Ab^1)$ is isomorphic to the quotient $[\End_\bfk(\bfk^n)/\GL_n]$, therefore its class in the Grothendieck group is 
\[ 
[\CCoh^n(\Ab^1)] = \frac{[\Ab^{n^2}]}{[\GL_n]}=
\frac{\Lb^{n^2}}{\displaystyle\prod_{0\leqslant i<n}(\BL^n-\BL^i)}.
\]
In particular, there is an identity\footnote{This can be seen as a refinement of Euler's formula
\[
\sum_{n\geqslant 0}\frac{t^n}{(1-q)\cdots (1-q^n)} = \prod_{k\geqslant 0}\left(1-q^kt\right)^{-1}.
\]
See \cite[Sec.~3.3]{BM15} for a full argument proving \eqref{eqn:refined-Euler}.}
\begin{equation}\label{eqn:refined-Euler}
\mathsf Z_{\BA^1}(t) = \prod_{k\geqslant 0} \left( 1-\BL^{-k}t\right)^{-1}.
\end{equation}
\end{prop}

\begin{proof}
Everything follows from the first sentence, which is a special case of the description of $\CCoh^n(\BA^d)$ as a quotient of the commuting variety \cite[Thm.~3.4]{Fantechi-Ricolfi-structural}.
\end{proof}

In dimension $2$, we have the Feit--Fine formula.

\begin{theorem}[{Feit--Fine \cite{FF1,BM15}}]
There is an identity
\[
\mathsf Z_{\BA^2}(t) = \prod_{m\geqslant 1}\prod_{k\geqslant 1}\,\left(1-\BL^{2-m}t^k \right)^{-1}.
\]
\end{theorem}

See \cite{MR18} for a module-theoretic, or geometric, interpretation of the first few coefficients of $\mathsf Z_{\BA^2}(t)$. 

\begin{remark}\label{rmk:DT_3}
No similar closed formula for $\mathsf Z_{\BA^d}(t)$ is currently available for any $d>2$. In fact, we do not even know whether the class of $\CCoh^n(\BA^d)$ is a polynomial in $\BL$, cf.~\Cref{open-problem-polynomial-in-L}.
\end{remark}

\smallbreak
As for Quot schemes, the situation is the following.
Fix integers $r,d>0$, a smooth $\bfk$-variety $X$ of dimension $d$ and a locally free sheaf $\CE$ of rank $r>0$ over it. Consider the generating series $\mathsf Q_\CE(t)$ from \eqref{eqn:motivic-series-quot-coh}.

We have the following known results.

\begin{theorem}[{\cite{BFP19,ricolfi2019motive}}]
If $C$ is a smooth curve, $\CE\in\Coh(C)$ is locally free of rank $r$, then $\Quot_C(\CE,n)$ is smooth and irreducible of dimension $nr$, and 
\[
\mathsf Q_\CE(t) = \Exp([C \times \BP^{r-1}]t) = \prod_{1\leqslant i \leqslant r} \zeta_C(\BL^{i-1}t).
\]
In particular, $\mathsf Q_\CE(t)$ is a rational function in $t$.
\end{theorem}

If $S$ is a smooth surface, then $\Quot_S(\CE,n)$ is singular as soon as $r>1$, but Ellingsrud--Lehn proved it is irreducible of dimension $n(r+1)$. The fibres of the Quot-to-Chow morphism $\mathsf q_{\CE,n}$ are irreducible as well \cite[Thm.~1]{Ellingsrud_Lehn}.

In the case $X=\BA^2$, the Quot scheme $\Quot_{\BA^2}(\OO_{\BA^2}^{\oplus r},n)$ sits inside the moduli space of framed sheaves on $\BP^2$ as a closed subscheme, cut out by a section of a tautological bundle. Mozgovoy proved the following motivic formula.

\begin{theorem}[Mozgovoy {\cite{mozgovoy2019motivic}}]
If $S$ is a smooth surface, $\CE\in\Coh(S)$ is locally free of rank $r$, then
\[
\mathsf Q_\CE(t) = \Exp\left(\frac{[S \times \BP^{r-1}]t}{1-\BL^rt}\right).
\]
\end{theorem}

As far as we know, not much is known if $d=\dim X>2$. For $d=3$, the situation is clearer for \emph{virtual} invariants: the Quot scheme $\Quot_{\BA^3}(\OO_{\BA^3}^{\oplus r},n)$ is the key player in higher rank Donaldson--Thomas theory of points \cite{Virtual_Quot,FMR_K-DT,cazzaniga2020higher}. In all dimensions $d>2$, the Quot scheme $\Quot_{\BA^d}(\OO_{\BA^d}^{\oplus r},n)$ is moreover isomorphic to the moduli space of framed sheaves on $\BP^d$ \cite{cazzaniga2020framed}.

We close this section proposing a (hard) open problem.

\begin{openproblem}\label{open-problem-polynomial-in-L}
Fix integers $r,d,n>0$. The following questions have positive answer for $d\leqslant 2$ (cf.~\Cref{sec:state-of-art}) but, to the best of our knowledge, are open for $d>2$.
\begin{itemize}
    \item [\mylabel{op-1}{(1)}] Is $[\Quot_{\BA^d}(\OO_{\BA^d}^{\oplus r},n)_0] \in K_0(\Var_\bfk)$ a polynomial in $\BL$?
    \item [\mylabel{op-2}{(2)}] Is $[\CCoh^n(\BA^d)_0] \in K_0(\St^{\aff}_{\bfk})$ a rational function in $\BL$?
\end{itemize}
\end{openproblem}

%%%%%%%%%%%%%%%%%%%%%%%%%%%%%%%%%%%
\appendix

\section{Failure of the fibration property}\label{app:et-vs-zar}

Let $n\ge 1$, $\alpha=(1^{\alpha_1}2^{\alpha_2}\cdots n^{\alpha_n})$ a partition of $n$, $X$ a smooth curve.
Then, by \Cref{prop:zar-triviality}\,\ref{punctual-3}, the Coh-to-Chow morphism $\CCoh^n_\alpha(X)\to \Sym^n_\alpha(X)$ is an \'etale locally trivial fibration with fibre 
\[
\prod_{i \,|\, \alpha_i\,\neq\, 0} \CCoh^i(\BA^1)_0^{\alpha_i},
\]
but the class of $\CCoh^n_\alpha(X)$ in the Grothendieck group is not the product of the class of $\Sym^n_\alpha(X)$ and the class of the fiber. We now provide an explicit example.

\begin{example} 
Set $X=\BA^1$, $n=2$, $\alpha=(1^2)$. 
We know by \Cref{cor:affine_line} that 
\[
\frac{\BL^4}{(\BL^2-1)(\BL^2-\BL)} =[\CCoh^2(\BA^1)] =  [\CCoh^2_{(2^1)}(\BA^1)] + [\CCoh^2_{(1^2)}(\BA^1)].
\]
A sheaf of length $2$ and support in one point is either $\kappa(x) \oplus \kappa(x)$, or the structure sheaf of $\Spec \bfk[t]/t^2$. In the former case, the automorphism group is $\GL_2$, in the latter it is $\GL_1 = \bfk^\times$. The support map
\[
\CCoh^2_{(2^1)}(\BA^1) \to \Sym^2_{(2^1)}(\BA^1) = \BA^1
\]
is \emph{trivial} by \Cref{prop:zar-triviality}\,\ref{punctual-1}, thus
\[
[\CCoh^2_{(2^1)}(\BA^1)] = \BL [\CCoh^2(\BA^1)_0] = \BL\left(\frac{1}{[\GL_2]} + \frac{1}{[\GL_1]}\right) = \frac{\BL + \BL (\BL+1)(\BL^2-\BL)}{(\BL^2-1)(\BL^2-\BL)}.
\]
It follows that 
\[
[\CCoh^2_{(1^2)}(\BA^1)] = \frac{\BL^2-\BL}{(\BL^2-1)(\BL^2-\BL)} = \frac{1}{\BL^2-1}.
\]
However, 
\[
[\Sym^2_{(1^2)}(\BA^1)] = [\Sym^2\BA^1] - [\Sym^2_{(2^1)}(\BA^1)] = \BL^2-\BL
\]
and
\[
[\CCoh^1(\BA^1)_0] = \frac{1}{\BL-1}
\]
but 
\[
[\Sym^2_{(1^2)}(\BA^1)]\cdot[\CCoh^1(\BA^1)_0]^2 =
\frac{\BL^2-\BL}{(\BL-1)^2} = \frac{\BL}{\BL-1} \neq \frac{1}{\BL^2-1} = [\CCoh^2_{(1^2)}(\BA^1)].
\]
\end{example}

\section{The Coh-to-Chow morphism over an affine base}\label{sec:supp-map}

Let $\CF$ be a coherent sheaf of dimension $0$ on a $\bfk$-variety $X$, with $\chi(\CF)=n$. Then
\[
\cycle(\CF) = \sum_{x \in \Supp(\CF)} \length_{\OO_{X,x}}\CF_x \cdot x
\]
is a closed point of $\Sym^n(X)$.
We now turn the association $\CF \mapsto \cycle(\CF)$ into a functorial operation.
The proof of the following result is due to Rydh \cite[IV, Prop.~7.8]{Rydh1} and works in \emph{much greater} generality. We give here a simplified account which is enough for the purpose of this paper.

\begin{theorem}[Rydh]\label{thm:coh-to-sym}
Let $S=\Spec A$, where $A$ is a ring containing $\BQ$ as a subring. Let $X \to S$ be a separated morphism of schemes. For any $n \in \BN$ there is an $S$-morphism
\[
\supp^n_{X/S} \colon \CCoh^n(X/S) \to \Sym^n(X/S)
\]
which, if $X$ is quasiprojective over $S=\Spec \bfk$, sends $[\CF] \in \CCoh^n(X)(\bfk)$ to the $0$-cycle $\cycle(\CF) \in \Sym^n(X)(\bfk)$.
\end{theorem}

The proof of this result needs some algebraic machinery, that we now introduce following Rydh \cite{Rydh1} and Roby \cite{Roby}. In what follows, we denote by $\Alg_A$ the category of algebras over a ring $A$. The symbol $\BN$ denotes the set of natural numbers $\set{0,1,2,\ldots}$.

\subsection*{Multiplicative polynomial laws}
Let $A$ be a ring, $M$ an $A$-module. Consider the covariant functor $\mathsf F_M \colon \Alg_A \to \Sets$ defined by $\mathsf F_M(A') = M \otimes_AA'$. If $N$ is another $A$-module, a \emph{polynomial law from $M$ to $N$} is a natural transformation $\eta \colon \mathsf F_M \Rightarrow \mathsf F_N$. We say that a polynomial law is \emph{homogeneous of degree $n \in \BN$} if for all $A' \in \Alg_A$ one has $\eta_{A'}(ax) = a^n\eta_{A'}(x)$ for all $a \in A'$ and for all $x \in M \otimes_AA'$, where $\eta_{A'} \colon M \otimes_AA' \to N \otimes_AA'$ is the map defined by $\eta$ on the object $A'$. We denote by $\Pol^n(M,N)$ the set of polynomial laws from $M$ to $N$ which are homogeneous of degree $n$.

\begin{example}
Any $A$-linear map $M \to N$ induces a polynomial law from $M$ to $N$. 
\end{example}

A canonical structure of $A$-module can be put on the set of polynomial laws from $M$ to $N$ \cite[p.~222]{Roby}, and moreover the set $\Pol^n(M,N)$ is an $A$-submodule \cite[Prop.~1.3]{Roby}.

A special case is when $M=B$ and $N=C$ are $A$-algebras themselves. In this case, a polynomial law $\eta$ from $B$ to $C$ is said to be \emph{multiplicative} if for every $A' \in \Alg_A$ the map $\eta_{A'} \colon B \otimes_AA' \to C \otimes_AA'$ preserves the ring identity and satisfies $\eta_{A'}(xy) = \eta_{A'}(x)\eta_{A'}(y)$ for every $x,y \in B \otimes_AA'$. We denote by $\Pol^n_{\mult}(B,C)$ the set of multiplicative polynomial laws from $B$ to $C$ which are homogeneous of degree $n$.

\subsection*{Symmetric algebra and symmetric tensors}
Let $A$ be a ring, $M$ an $A$-module. Denote by $\Ten_A(M) = \bigoplus_{n \geqslant 0} \Ten^n_A(M)$ the tensor algebra of $M$ (with algebra structure given by the tensor product $\otimes_A$). Consider the two-sided ideal $I\subset \Ten_A(M)$ generated by elements $x\otimes_Ay - y \otimes_Ax$ for $x, y \in \Ten_A(M)$. 
The quotient
\[
\Sym_A(M) = \Ten_A(M)/I
\]
is the \emph{symmetric algebra} of $M$, equipped with its natural grading
\[
\Sym_A(M) = \bigoplus_{n\geqslant 0}\Sym^n_A(M).
\]
We have $\Ten^0_A(M) = A = \Sym_A^0(M)$ and $\Ten^1_A(M) = M = \Sym^1_A(M)$. 
There is a graded $A$-submodule $\TS_A(M) \subset \Ten_A(M)$, called the $A$-\emph{module of symmetric tensors}, defined by
\[
\TS_A(M) = \bigoplus_{n \geqslant 0} \TS_A^n(M), \quad \TS_A^n(M) = \Ten^n_A(M)^{\FS_n} \subset \Ten^n_A(M),
\]
where the symmetric group $\FS_n$ acts on a homogeneous tensor of degree $n$ by 
\[
\sigma(x_1 \otimes_A \cdots \otimes_A x_n) = x_{\sigma(1)} \otimes_A \cdots \otimes_A x_{\sigma(n)}.
\]
The $A$-module $\TS_A(M)$ has a commutative, associative $A$-algebra structure given by the \emph{shuffle product} (see \cite[p.~253]{Roby} or \cite[I \S (1.1.5)]{Rydh1}).

\subsection*{Divided powers}
Let $A$ be a ring, $M$ an $A$-module. There is a graded $A$-algebra 
\begin{equation}\label{eqn:div-power-algebra}
\Gamma_A(M) = \bigoplus_{n \geqslant 0}\Gamma^n_A(M)
\end{equation}
called the \emph{algebra of divided powers}, satisfying $\Gamma_A^0(M) = A$ and $\Gamma_A^1(M) = M$. We now quickly review its definition following \cite[Ch.~III]{Roby}.

Given $(x,n) \in M \times \BN$, introduce a formal variable $\XX_{(x,n)}$. Form the polynomial $A$-algebra $G = A[\XX_{(x,n)}\,|\,(x,n) \in M \times \BN]$. Consider the collection of polynomials
\begin{equation}\label{eqn:div-powers-relations}
    \begin{split}
        &\XX_{(x,0)}-1 \\
        &\XX_{(\lambda x,n)}-\lambda^n \XX_{(x,n)} \\
        &\XX_{(x,m)}\XX_{(x,n)} - ((m,n)) \XX_{(x,m+n)} \\
        &\XX_{(x+y,n)}-\sum_{0\leqslant i\leqslant n}\XX_{(x,i)}\XX_{(y,n-i)}
    \end{split}
\end{equation}
where $x,y \in M$, $\lambda \in A$ and $m,n \in \BN$. In the third equation, we have set $((m,n)) = \binom{m+n}{m}$.
Let $J \subset G$ be the ideal generated by the polynomials in \eqref{eqn:div-powers-relations}. Set
\[
\Gamma_A(M) = G/J.
\]
Denote by $x^{[n]} \in \Gamma_A(M)$ the image of $\XX_{(x,n)} \in G$ along the quotient map. In $\Gamma_A(M)$, the generators $x^{[n]}$ satisfy the relations
\begin{equation}\label{eqn:div-powers-generators-relations}
    \begin{split}
       x^{[0]} &= 1 \\
       (\lambda x)^{[n]} &= \lambda^n x^{[n]} \\
       x^{[m]}x^{[n]} &= ((m,n)) x^{[m+n]} \\
       (x+y)^{[n]} &= \sum_{0\leqslant i\leqslant n}x^{[i]}y^{[n-i]}.
    \end{split}
\end{equation}
There is a unique $\BN$-grading on $G$, compatible with its $A$-algebra structure, for which $\XX_{(x,n)}$ has degree $n$. With respect to this grading, the ideal $J$ is homogeneous, thus the quotient $\Gamma_A(M) = G/J$ inherits a natural $\BN$-grading, namely \eqref{eqn:div-power-algebra} above, for which $x^{[n]}$ has degree $n$. Each operation $(-)^{[n]}$ is a natural map
\[
\begin{tikzcd}
M \arrow{r}{\gamma^n} & \Gamma_A^n(M), \quad x \mapsto x^{[n]}.
\end{tikzcd}
\]

\smallbreak
Consider the canonical injection $\gamma^1\colon M \into \Gamma_A(M)$ defined by sending $x \mapsto x^{[1]}$. We write $x$ instead of $x^{[1]}$ when we view it inside the algebra of divided powers. The third relation in \eqref{eqn:div-powers-generators-relations} readily implies $xx^{[n-1]} = nx^{[n]}$ for all $n \geqslant 1$, which in turn yields, by induction,
\[
x^n = n! x^{[n]}, \quad n \in \BN.
\]
In particular, if $A \supset \BQ$, one can make sense of the relation $x^{[n]} = x^n / n!$ in $\Gamma_A(M)$.

\subsection*{Divided powers and symmetric algebra}
The canonical injection $\gamma^1\colon M \into \Gamma_A(M)$ extends to an $A$-algebra homomorphism $\theta \colon \Sym_A(M) \to \Gamma_A(M)$ sending $x^{\vee n} \mapsto x^n$, where `$\vee$' denotes the multiplication on $\Sym_A(M)$. On the other hand, if $A \supset \BQ$, one has the $A$-algebra homomorphism $G \to \Sym_A(M)$ sending $\XX_{(x,n)} \mapsto x^{\vee n} / n!$. This descends to the quotient, giving rise to a homomorphism $\zeta \colon \Gamma_A(M) \to \Sym_A(M)$ sending $x^{[n]} \mapsto x^{\vee n} / n!$. The homomorphisms $\theta$ and $\zeta$ are inverse to each other. Therefore
\[
\Gamma_A(M) \cong \Sym_A(M)
\]
as soon as $A \supset \BQ$.

\subsection*{Divided powers and symmetric tensors}

As proved in \cite[Prop.~III.1, p.~254]{Roby}, there is one and only one $A$-algebra homomorphism 
\[
\rho \colon \Gamma_A(M) \to \TS_A(M)
\]
sending $x^{[n]} \mapsto x^{\otimes_An}$. Let $q \colon \Ten_A(M) \to \Sym_A(M)$ be the quotient map. Its restriction $\TS_A(M) \into \Ten_A(M) \to \Sym_A(M)$ is \emph{not} an $A$-algebra homomorphism. However, as observed in \cite[Prop.~III.3, p.~256]{Roby}, making a homogeneous element $z \in \Sym_A^n(M)$ take a tour around the diagram
\[
\begin{tikzcd}
& \Gamma_A(M)\arrow{dr}{\rho} & \\
\Sym_A(M)\arrow{ur}{\theta} & & \TS_A(M) \arrow{ll}{q}
\end{tikzcd}
\]
produces the same element $z$ multiplied by $n!$. One deduces an isomorphism of $A$-algebras
\[
\Gamma_A(M) \cong \TS_A(M)
\]
as soon as $A \supset\BQ$.

\subsection*{Divided powers and polynomial laws}
Sending $M \mapsto \Gamma_A(M)$ is a functor from $A$-modules to graded $A$-algebras \cite[Ch.~III \S 4, p.~251]{Roby}. There are base change isomorphisms 
\[
\begin{tikzcd}
\Gamma_A(M) \otimes_AA' \arrow{r}{\sim} & \Gamma_{A'}(M \otimes_AA'), \quad A' \in \Alg_A
\end{tikzcd}
\]
sending $x^{[n]} \otimes_A 1 \mapsto (x \otimes_A1)^{[n]}$ (cf.~\cite[Thm.~III.3, p.~262]{Roby}), therefore showing that $\gamma^n \in \Pol^n(M,\Gamma_A^n(M))$. In fact, one has functorial bijections
\[
\begin{tikzcd}
\Hom_A(\Gamma^n_A(M),N) \arrow{r}{\sim} & \Pol^n(M,N), \quad h \mapsto h \circ \gamma^n
\end{tikzcd}
\]
for every $n \geqslant 0$ \cite[Thm.~IV.1, p.~266]{Roby}, where `$\circ$' is the natural composition of polynomial laws. This can be interpreted as a universal property for $\Gamma_A(M)$. 

Given $A$-modules $M$ and $N$, one always has a multiplication map \cite{Roby}
\[
\Gamma_A^n(M) \otimes_A \Gamma^n_A(N) \to \Gamma_A^n(M\otimes_AN),
\]
sending $x^{[n]} \otimes y^{[n]} \mapsto (x \otimes_A y)^{[n]}$.

Let $B$ be an $A$-algebra with multiplication $B\otimes_AB \to B$. 
There is an $A$-algebra structure on $\Gamma_A^n(B)$ obtained as follows: one can exploit functoriality of $\Gamma^n_A(-)$ to form the composition
\[
\Gamma_A^n(B) \otimes_A \Gamma^n_A(B) \to \Gamma_A^n(B\otimes_AB) \to \Gamma_A^n(B).
\]
This produces an $A$-algebra structure on $\Gamma_A^n(B)$, with ring identity $\gamma^n(1)$, satisfying $x^{[n]}y^{[n]} = (xy)^{[n]}$ for all $x,y \in B$. This relation turns $\gamma^n \colon B \to \Gamma_A^n(B)$ into a multiplicative polynomial law of homogeneous degree $n$. If $B$ and $C$ are $A$-algebras, we have the universal property \cite{Roby}
\[
\begin{tikzcd}
\Hom_{\Alg_A}(\Gamma^n_A(B),C) \arrow{r}{\sim} & \Pol_{\mult}^n(B,C), \quad h \mapsto h \circ \gamma^n.
\end{tikzcd}
\]

\subsection*{The scheme of divided powers}
Here we follow \cite[I, \S 1.4]{Rydh1}. Let $S$ be a scheme, $\CA$ a quasicoherent sheaf of $\OO_S$-algebras.
A quasicoherent sheaf of $\OO_S$-algebras $\Gamma^n_{\OO_S}(\CA)$ may be constructed thanks to the fact that $\Gamma_A(-)$ commutes with all localisations of $A$. If $f \colon X \to S$ is an affine morphism, one defines the $S$-scheme
\[
\Gamma^n(X/S) = \Spec_{\OO_S} \Gamma^n_{\OO_S}(f_\ast \OO_X).
\]
The operation $\Gamma^n(-)$ is an endofunctor on schemes affine over $S$.

For any affine $S$-scheme $T$, we have
\begin{equation}\label{eqn:div-pow-basechange}
\begin{split}
    \Hom_{\Sch_S}(T,\Gamma^n(X/S)) 
    &= \Hom_{\Sch_T}(T,\Gamma^n(X/S) \times_ST)\\
    &=\Hom_{\Sch_T}(T,\Gamma^n(X \times_ST/T)).
\end{split}
\end{equation}
In particular, if $X=\Spec B$ and $S=\Spec A$, then
\[
\Gamma^n(X/S) = \Spec \Gamma^n_A(B).
\]
If $A \supset \BQ$, since $(B^{\otimes_An})^{\FS_n} \cong \Gamma^n_A(B)$ as $A$-algebras, we have
\[
\Gamma^n(X/S) = \Spec\,(B^{\otimes_An})^{\FS_n} = \Sym^n(X/S).
\]

\subsection*{Families of 0-cycles}
We review Rydh's functor of families of cycles from \cite{Rydh1}, limiting ourselves to less general assumptions. In particular, we only talk about schemes and not algebraic spaces.

\begin{definition}[{\cite[I, Def.~3.1.1]{Rydh1}}]
Let $S$ be a scheme, $f \colon X\to S$ a separated morphism of schemes, locally of finite type (e.g.~a quasiprojective morphism). A \emph{family of 0-cycles of degree} $n$ relative to $f$ is a pair $(Z,\alpha)$, where $Z \into X$ is a closed subscheme such that $Z \to S$ is integral, and $\alpha \colon S \to \Gamma^n(Z/S)$ is a morphism. We have an equivalence relation $\sim$ on such pairs, given by declaring that $(Z_1,\alpha_1) \sim (Z_2,\alpha_2)$ when there is a closed subscheme $Z \into X$, admitting closed immersions $j_i \colon Z \into Z_i$, and a morphism $\alpha \colon S \to \Gamma^n(Z/S)$, such that $\alpha_i$ is the composition $j_{i\ast} \circ \alpha \colon S \to \Gamma^n(Z/S) \into \Gamma^n(Z_i/S)$ for $i = 1,2$.
\end{definition}

Rydh \cite[I, Def.~3.1.3]{Rydh1} defines a functor 
\[
\underline\Gamma^n_{X/S} \colon \Sch_S^{\opp} \to \Sets
\]
sending $T \to S$ to the set of equivalence classes of families of 0-cycles of degree $n$ relative to $X \times_ST \to T$. He proves the following result.

\begin{prop}[{\cite[I, Thm.~3.4.1,\,Cor.~4.2.5]{Rydh1}}]
Let $S$ be a scheme, $f \colon X\to S$ a separated morphism of schemes, locally of finite type. Then $\underline\Gamma^n_{X/S}$ is represented by a separated $S$-scheme $\Gamma^n(X/S)$, which coincides with $\Spec_{\OO_S} \Gamma^n_{\OO_S}(f_\ast \OO_X)$ if $f$ is affine. If $S$ has pure characteristic 0 or $X \to S$ is flat, there is a canonical isomorphism
\[
\begin{tikzcd}
    \Sym^n(X/S)\arrow{r}{\sim} & \Gamma^n(X/S).
\end{tikzcd}
\]
\end{prop}

\subsection*{The norm family}
Let $A$ be a ring. Fix an $A$-algebra $A \to B$ and a $B$-module $M$ which is free of rank $n$ as an $A$-module. The norm map associated to the data $(A \to B,M)$ is the map
\begin{equation}\label{det-law}
\begin{tikzcd}
\mathrm{N}_{B/A} \colon B \arrow{r} & \End_A(M) \arrow{r}{\det} & \End_A(\wedge^n M) = A
\end{tikzcd}
\end{equation}
where the first \emph{homomorphism} takes $b \in B$ to the endomorphism $M \to M$ sending $m \mapsto bm$, and the second \emph{map} is defined by
\[
\det(\phi) (m_1\wedge \cdots \wedge m_n) = \phi(m_1) \wedge \cdots \wedge \phi(m_n),\quad \phi \in \End_A(M).
\]
The construction of the \emph{determinant law} \eqref{det-law} can be extended to $M$ locally free by passing to an open cover of $\Spec A$. See \cite[IV, \S (7.5)]{Rydh1} for more details.

Set $S=\Spec A$ and $X=\Spec B$. Let $X \to S$ be the morphism corresponding to $A \to B$. The crucial observation, now, is that the determinant law is a multiplicative law of homogeneous degree $n$, and so it defines an element
\begin{align*}
\mathrm{N}_{B/A}(M) \,\,\in\,\, \Pol^n_{\mult}(B,A) &= \Hom_{\Alg_A}(\Gamma^n_A(B),A) \\
&= \Hom_{\Sch_S}(S,\Spec \Gamma^n_A(B)) \\
&= \Hom_{\Sch_S}(S,\Sym^n(X/S)).    
\end{align*}

\subsection*{Construction of the support map over an affine base}

We finally have all the tools to construct the support map.

\begin{theorem}[Rydh]
Let $S=\Spec A$, where $A$ is a ring containing $\BQ$ as a subring. Let $X \to S$ be a separated morphism of schemes. For any $n \in \BN$ there is an $S$-morphism
\[
\supp^n_{X/S} \colon \CCoh^n(X/S) \to \Sym^n(X/S)
\]
which, if $X$ is quasiprojective over $S=\Spec \bfk$, sends $[F] \in \CCoh^n(X)(\bfk)$ to the $0$-cycle $[\cycle(F)] \in \Sym^n(X)(\bfk)$.
\end{theorem}

\begin{proof}
We divide the construction of the morphism in three steps.

\smallbreak
\noindent
\textbf{Step I. Reduction to affine base.}
To construct the sought after morphism, it is enough to construct it on affine $S$-schemes $T=\Spec R$. This is because the $\CCoh^n(X/S)$ is locally of finite presentation, hence limit preserving, and moreover any scheme is a colimit of affine schemes.

\smallbreak
\noindent
\textbf{Step II. Reduction to affine source.}
Let $\cF \in \Coh(X)$ be an $S$-flat family of sheaves of relative Hilbert polynomial $n$. Let $Z=\Supp(\cF)$ be its support, with inclusion $\iota \colon Z \into X$. Then $Z \to S$ is finite and $\iota_\ast \iota^\ast \cF = \cF$. Therefore after replacing $X$ with $Z$ we may assume $X \to S$ is finite, in particular affine. But $S$ is affine, therefore we may assume $X$ is affine.

\smallbreak
\noindent
\textbf{Step III. The affine case.}
Set $X=\Spec B$ and $S=\Spec A$. As already observed, we have
\[
\Sym^n(X/S) = \Spec \Gamma^n_A(B) = \Gamma^n(X/S).
\]
Let $T=\Spec R \to S$ be an affine $S$-scheme. Fix a $T$-valued point of $\CCoh^n(X/S)$, i.e.~a finitely presented $(B \otimes_AR)$-module $M$ such that $M$ is locally free of rank $n$ as an $R$-module (via the map $R \to B \otimes_AR$ corresponding to $X \times_ST\to T$). Consider the norm family  associated to $(R \to B\otimes_AR,M)$. It defines an element
\begin{equation}
\begin{split}
    \mathrm{N}_{B\otimes_AR/R}(M)\,\,\in\,\,\Pol^n_{\mult}(B\otimes_AR,R)
    &= \Hom_{\Alg_R} (\Gamma^n_R(B\otimes_AR),R) \\
    &= \Hom_{\Sch_T}(T,\Spec \Gamma^n_R(B\otimes_AR)) \\
    &= \Hom_{\Sch_T}(T,\Gamma^n(X\times_ST / T)) \\
    &= \Hom_{\Sch_S}(T,\Gamma^n(X/S)) \\
    &= \Hom_{\Sch_S}(T,\Sym^n(X/S)).
\end{split}
\end{equation}
This yields the sought after $T$-valued point of the $S$-scheme $\Sym^n(X/S)$.
\end{proof}

%%%%%%%%%%%%%%%%%%%%%%%%%%%%%%%%%%%%%%%%%%%%%%%%%%%%%%%%%%%%%%%%%%%%%%%%%%%%%%%%%%%%%%%%%%%%%
\section{Basic theory of algebraic stacks}\label{sec:stacks}

We include this appendix in order to make the paper self-contained. We refer the reader to \cite{MR2223406} for a thorough introduction to stacks in full generality.

\subsection{Groupoid fibrations and the definition of stacks}
Let $B$ a scheme. We denote by $\Sch_B$ the category of $B$-schemes, whose objects are pairs $(U,q)$ where $U$ is a scheme and $q\colon U \to B$ is a morphism of schemes. We shall abuse notation by simply writing $U$ for an object of this category. The \'etale coverings $\set{U_i \to U}_{i \in I}$ of objects $U \in \Ob(\Sch_B)$ define a Grothendieck topology $\tau_{\textrm{\'et}}$ and thus a \emph{site} 
\[
\Sch_{B,\textrm{\'et}} = (\Sch_B,\tau_{\textrm{\'et}})
\]
called the (big) \'etale site of $B$-schemes. Fix a $B$-scheme $U \in \Ob(\Sch_B)$ and a functor $\mathsf F \colon \Sch_U^{\opp} \to \Sets$. We say that $\mathsf F$ is an \emph{\'etale sheaf} if it satisfies the sheaf axioms on the \'etale site $\Sch_{U,\textrm{\'et}}$, which means that for every $V \in \Ob(\Sch_U)$ and for every \'etale cover $\set{V_i \to V}_{i \in I} \in \tau_{\textrm{\'et}}$ the diagram
\[
\begin{tikzcd}
    \mathsf F(V) 
\arrow[r] 
&
\displaystyle\prod_{i\in I}\mathsf F(V_i) 
\arrow[r, shift left]
\arrow[r, shift right] 
& 
\displaystyle\prod_{(i,j) \in I\times I} \mathsf F(V_i \times_V V_j)
\end{tikzcd}
\]
is an equaliser in the category of sets.

Given a category $\CX$, a covariant functor $p\colon \CX \to \Sch_B$ and an object $U\in\Ob(\Sch_B)$, the \emph{fibre category} over $U$ is the subcategory $p^{-1}(U)\subset \CX$ which has, as objects, the objects $\eta \in \Ob(\CX)$ such that $p(\eta) = U$ and, as morphisms between two objects $\eta$ and $\eta'$, the morphisms $f\colon \eta \to \eta'$ in $\CX$ such that $p(f) = \id_U$. A \emph{groupoid fibration} over $B$ is a covariant functor $p\colon \CX \to \Sch_B$ such that the fibre category $p^{-1}(U)$ is a groupoid for every $U \in \Ob(\Sch_B)$, which means that every arrow in $p^{-1}(U)$ is an isomorphism.

If $p \colon \CX \to \Sch_B$ is a covariant functor, we say (cf.~\cite[Def.~3.1]{MR2223406}) that an arrow $\phi \in \Hom_{\CX}(\xi,\eta)$ is \emph{cartesian} if for any arrow $\psi \in \Hom_{\CX}(\zeta,\eta)$ and for any arrow $h \in \Hom_{\Sch_B} (p(\zeta),p(\xi))$ such that $p(\phi) \circ h = p(\psi)$, there is a unique arrow $\theta \in \Hom_\CX(\zeta,\xi)$ such that $p(\theta) = h$ and $\phi\circ \theta=\psi$.
This is depicted in the following diagram.

\[
\begin{tikzcd}[column sep=large]
    \zeta\arrow[mapsto]{dd}\arrow[bend left=20]{drr}{\forall\,\psi}\arrow[dotted]{dr}[description]{\exists\,!\,\theta} & & \\
    & \xi\arrow[mapsto,crossing over]{dd}\arrow{r}{\phi} & \eta\arrow[mapsto]{dd} \\
    p(\zeta)\arrow[bend left=20]{drr}[description]{p(\psi)}\arrow{dr}[description]{\forall\,h} & & \\
    & p(\xi)\arrow{r}{p(\phi)} & p(\eta)
\end{tikzcd}
\]

Given a groupoid fibration $p \colon \CX \to \Sch_B$, we always assume a `choice of pullbacks' has been made once and for all: this means that for every morphism $h\colon V \to U$ in the base category $\Sch_B$ and for every object $\eta \in \Ob(p^{-1}(U))$, a choice of a cartesian arrow $h^\ast \eta \to \eta$ in $\CX$ such that $p(h^\ast \eta \to \eta) = h$, has been made once and for all. 
\[
\begin{tikzcd}
    & \eta\arrow[mapsto]{d}{p} \\
    V \arrow{r}{h} & U
\end{tikzcd} \qquad \qquad \rightsquigarrow \qquad \qquad
\begin{tikzcd}[column sep=large]
    h^\ast \eta\arrow{r}{\textrm{cartesian}} \arrow[mapsto]{d}{p} & \eta\arrow[mapsto]{d}{p} \\
    V \arrow{r}{h} & U
\end{tikzcd}
\]
Such a choice translates into the datum of a `pullback functor' $h^\ast \colon p^{-1}(U) \to p^{-1}(V)$, for every $h\in\Hom_{\Sch_B}(V,U)$, realising a canonical isomorphism
\[
k^\ast \circ h^\ast \cong (h \circ k)^\ast
\]
for every pair of composable arrows $h \colon V \to U$ and $k \colon W \to V$. In fact, this choice of pullbacks is sometimes incorporated in the definition of a groupoid fibration, cf.~Definitions 3.5 and 3.21 in \cite{MR2223406}. 

\begin{notation}\label{notation_pullback}
Given a groupoid fibration $p \colon \CX \to \Sch_B$ (with a choice of pullbacks as above), and an object $\eta \in \Ob(p^{-1}(U))$, we shall use the notation $\eta|_V$ to denote the object $h^\ast \eta \in \Ob(p^{-1}(V))$.
\end{notation}

Let $p\colon \CX \to \Sch_B$ be a groupoid fibration (with a choice of pullbacks as above).
Given $U \in \Ob(\Sch_B)$ and two objects $\eta,\eta' \in \Ob(p^{-1}(U))$ in the fibre category, one has a presheaf 
\[
\begin{tikzcd}
\Iso_{\CX}(\eta,\eta')\colon \Sch_U^{\mathrm{op}} \arrow{r} & \Sets
\end{tikzcd}
\]
defined by sending an object $h\colon V \to U$ to the set $\Hom_{p^{-1}(V)}(h^\ast \eta,h^\ast\eta')$. Note that in this set all arrows are isomorphisms, by the groupoid condition, whence the notation.

Given an \'etale covering $\CU = \set{U_i \to U}_{i \in I}$ of an object $U \in \Ob(\Sch_B)$, a \emph{descent datum} relative to $\CU$ consists of the following data: an object $\eta_i \in \Ob(p^{-1}(U_i))$ for every $i \in I$, along with a collection of isomorphisms $\varepsilon_{ij}\colon \eta_i|_{U_{ij}} \simto \eta_j|_{U_{ij}}$ in $p^{-1}(U_{ij})$, where $U_{ij} = U_i\times_UU_j$, satisfying the cocycle condition on triple intersections. A descent datum $(\eta_i,\varepsilon_{ij})_{i,j}$ relative to $\CU$ is said to be \emph{effective} if there is a global object $\eta \in \Ob(p^{-1}(U))$ along with an isomorphism $\alpha_i\colon \eta|_{U_i} \simto \eta_i$ in $p^{-1}(U_i)$ for each $i$, such that $\varepsilon_{ij} = \alpha_j|_{U_{ij}} \circ (\alpha_i|_{U_{ij}})^{-1}$ for every $(i,j) \in I\times I$.

\begin{definition}\label{def:stack}
A $B$-\emph{stack}, or a stack over $B$, is a groupoid fibration $p\colon \CX \to \Sch_B$ (equipped with a choice of pullbacks) such that 
\begin{enumerate}
    \item For every object $U \in \Ob(\Sch_B)$ and for every $\eta,\eta' \in \Ob(p^{-1}(U))$, the presheaf $\Iso_{\CX}(\eta,\eta')$ is an \'etale sheaf, and
    \item For every object $U \in \Ob(\Sch_B)$ and for every \'etale cover $\CU = \set{U_i \to U}$, all descent data $(\eta_i,\varepsilon_{ij})_{i,j}$ relative to $\CU$ are effective.
\end{enumerate}
\end{definition}

\subsection{Morphisms of stacks and their fibre products}
A morphism of $B$-stacks is a functor $f \colon \CX \to \CX'$ between the corresponding categories, respecting the projection to $\Sch_B$. These are called $1$-\emph{morphisms}, and they form themselves a category for every fixed $\CX$ and $\CX'$. Morphisms between $1$-morphisms are called 2-morphisms (strictly speaking, they are natural transformations). Stacks thus form a 2-category, in the sense of \cite[\href{https://stacks.math.columbia.edu/tag/003G}{Tag 003G}]{stacks-project}.

Let $p \colon \CX \to \Sch_B$ be a $B$-stack. The functor $p$ is called the structure morphism of the $B$-stack $\CX$. Sometimes we shall use the notation $\CX \to B$ to denote it. This is legitimate, because every $B$-scheme $X$ (e.g.~$B$ itself) is naturally a $B$-stack: just consider the functor $\Sch_X \to \Sch_B$. A $B$-stack arising in this fashion (i.e.~from a scheme) is called \emph{representable}.

\begin{remark}\label{rmk:1-mor-objects}
Let $p\colon \CX \to B$ be a $B$-stack, $U$ a $B$-scheme. To give a $1$-morphism $U \to \CX$ is the same as to give an object $\eta$ in the fibre category $p^{-1}(U)$. We shall use the (common) notation $\CX(U)$ to denote the fibre category $p^{-1}(U)$. We call the objects of this groupoid the $U$-\emph{valued points of} $\CX$. Informally speaking, one may think of a stack as an association
\[
\begin{tikzcd}
U \arrow[mapsto]{r} & \CX(U),
\end{tikzcd}
\]
where $\CX(U)$ is a groupoid, such that objects `glue nicely' along \'etale coverings. In other words, a $B$-stack is a `sheaf of groupoids' on the big \'etale site of $B$-schemes (one can also pick different topologies, but we will not do that here).
\end{remark}

\begin{definition}\label{def:fibre-product}
An \emph{isomorphism} of $B$-stacks is a pair $(f,g)$ of $1$-morphisms $f\colon \CX \to \CX'$ and $g \colon \CX' \to \CX$ over $B$ such that $g\circ f$ is 2-isomorphic to $\id_{\CX}$ and $f\circ g$ is 2-isomorphic to $\id_{\CX'}$.
\end{definition}

\begin{remark}
A morphism of $B$-stacks $\CX \to \CX'$ is an isomorphism if and only if the induced functor $\CX(U) \to \CX'(U)$ is an equivalence of groupoids for every $B$-scheme $U$.
\end{remark}

\begin{definition}\label{def:fibre-product-stacks}
Suppose given two morphisms of $B$-stacks $f \colon \CX \to \CY$ and $g \colon \CZ \to \CY$. The \emph{fibre product} of $f$ and $g$ is the category $\CX \times_{\CY}\CZ$ whose objects are triples $(x,z,\alpha)$, with $x \in \Ob(\CX)$, $z \in \Ob(\CZ)$ and $\alpha \colon f(x) \simto g(z)$ an isomorphism in a fibre of $\CY \to \Sch_B$. A morphism $(x,z,\alpha) \to (x',z',\alpha')$ is a pair
\[
\left(\beta_1\colon x \to x', \beta_2\colon z \to z'\right)
\]
of morphisms in fibre categories of $\CX \to \Sch_B$ and $\CZ \to \Sch_B$ respectively, such that 
\[
g(\beta_2) \circ \alpha = \alpha' \circ f(\beta_1) \in \Hom_{\CY}(f(x),g(z')).
\]
\end{definition}

The fibre product $\CX \times_\CY\CZ$ is naturally a $B$-stack. In the situation of \Cref{def:fibre-product-stacks}, one has a diagram
\[
\begin{tikzcd}
    \CX\arrow{d}\arrow{r} \times_{\CY}\CZ & \CX\arrow{d}{f} \\
    \CZ\arrow{r}{g} & \CY
\end{tikzcd}
\]
which is universally 2-commutative.
Analogously to the case of schemes, one can then define the diagonal morphism
\[
\begin{tikzcd}
\Delta_f \colon \CX \arrow{r} & \CX \times_{\CY}\CX
\end{tikzcd}
\]
for every $1$-morphism $f \colon \CX \to \CY$ of stacks. Explicitly, when evaluated on a $B$-scheme $U$, this morphism sends $\eta \in \Ob (\CX(U))$ to the triple $(\eta,\eta,\id_{f(\eta)}) \in \Ob((\CX \times_{\CY}\CX)(U))$.

\begin{definition}[Representable morphism]
A morphism of stacks $\CX \to \CY$ is \emph{schematic} (resp.~\emph{representable by algebraic spaces}) if for every morphism $S \to \CY$ from a scheme the fibre product $S \times_{\CY}\CX$ is a scheme (resp.~an algebraic space). (The definition of algebraic space is postponed to \Cref{rmk:alg-space}).
\end{definition}

Let $\mathscr P$ be a property of morphisms of algebraic spaces that is stable under base change and \'etale local on the target. Then a representable morphism of $B$-stacks $\CX \to \CY$ is said to have property $\mathscr P$ if for any morphism $S \to \CY$ from a scheme, the induced morphism of schemes $S \times_{\CY}\CX \to S$ has property $\mathscr P$. This includes properties like: surjective, quasicompact, closed immersion, open immersion, (locally) of finite type, separated, proper, flat, smooth, unramified, \'etale.

\begin{definition}\label{def:alg-dm-atlas}
Let $B$ be a scheme. A $B$-stack $\CX \to B$ is \emph{algebraic} (resp.~\emph{of Deligne--Mumford type}) if the following conditions are satisfied:
\begin{itemize}
    \item [(1)] The diagonal $\CX \to \CX\times_B\CX$ is representable by algebraic spaces, separated and quasicompact, and
    \item [(2)] there is a scheme $Y$ and a smooth (resp.~\'etale) surjective $1$-morphism $Y \to \CX$.
\end{itemize} 
The morphism $Y \to \CX$, or simply the scheme $Y$, is called an \emph{atlas} of the stack.
\end{definition}

\begin{definition}\label{def:stratification}
Let $\CX$ and $\CY$ be algebraic stacks locally of finite type over an algebraically closed field $\bfk$. A 1-morphism $f \colon \CX \to \CY$ is a \emph{geometric bijection} if it induces an equivalence 
    \[
    \begin{tikzcd}
    f(\bfk) \colon    \CX(\bfk) \arrow{r}{\sim} & \CY(\bfk)
    \end{tikzcd}
    \]
between the groupoids of $\bfk$-points. Given locally closed substacks $\set{\CZ_i \into \CY}_{i \in I}$ of an algebraic stack $\CY$, we say that $\set{\CZ_i}_{i \in I}$ form a \emph{stratification} of $\CY$ if the immersions induce a geometric bijection $\coprod_{i \in I}\CZ_i \to \CY$. Given this data, we simply write $\CY = \coprod_{i \in I}\CZ_i$, with a slight abuse of notation.
\end{definition}

\begin{example}
Let $Y$ be a $\bfk$-scheme, $G$ an algebraic group (necessarily smooth over $\bfk$ by Cartier's theorem) and $\sigma \colon Y \times_\bfk G \to Y$ a $G$-action on $Y$. For a $\bfk$-scheme $U$, consider the category $[Y/G](U)$ whose objects are pairs $(\pi,f)$ represented as diagrams
\[
\begin{tikzcd}
P \arrow[swap]{d}{\pi}\arrow{r}{f} & Y \\
U &
\end{tikzcd}
\]
where $\pi$ is a principal $G$-bundle with base $U$ and $f$ is a $G$-equivariant morphism, and where a morphism $(\pi,f) \to (\pi',f')$ is defined to be a morphism $\phi\colon P \to P'$ such that $f'\circ \phi=f$ and $\pi'\circ \phi= \pi$. It is a classical fact that $\phi$ must be an isomorphism, which confirms that $[Y/G](U)$ is a groupoid. The association
\[
\begin{tikzcd}
U \arrow[mapsto]{r} & {[}Y/G{]}(U)
\end{tikzcd}
\]
defines an algebraic stack $[Y/G]$ over $\bfk$. See \cite{LMB} for more details.
The diagram
\[
\begin{tikzcd}
Y \times_\bfk G \arrow[swap]{d}{\mathrm{pr}_1}\arrow{r}{\sigma} & Y \\
Y &
\end{tikzcd}
\]
defines an atlas $Y \to [Y/G]$. 

When $Y=\Spec \bfk$, the resulting stack $\mathrm BG=[\Spec \bfk/G]$ is called the \emph{classifying stack} of principal $G$-bundles.
\end{example}

\subsection{Inertia stack and stabiliser groups}
Let $p\colon \CX \to B$ be an algebraic stack. The 2-fibre product $\CI_\CX$ of the diagonal $\CX \to \CX \times_B\CX$ with itself is again an algebraic stack. It is called the \emph{inertia stack of} $\CX$. It comes equipped with a natural $1$-morphism
\[
\begin{tikzcd}
\CI_\CX \arrow{r} & \CX,
\end{tikzcd}
\]
which is representable by algebraic spaces and locally of finite type. See \cite[\href{https://stacks.math.columbia.edu/tag/050P}{Tag 050P}]{stacks-project} for a proof of all these assertions.
Roughly speaking, the proof goes as follows.
One observes that the objects of $\CI_\CX$ are pairs $(\eta,\alpha)$ where $\eta$ is an object of $\CX$ and $\alpha$ is an automorphism of $\eta$ in the fiber category $\CX(U)=p^{-1}(U)$, where $U = p(\eta)$.
One has the fibre diagram
\[
\begin{tikzcd}
\Iso_{\CX}(\eta,\eta)\MySymb{dr} \arrow{r}\arrow{d} & \CI_\CX \arrow{d} \\ U \arrow[swap]{r}{\eta} & \CX 
\end{tikzcd}
\]
where we denote by $\eta \colon U \to \CX$ is the $1$-morphism corresponding to $\eta \in \Ob(\CX(U))$, cf.~\Cref{rmk:1-mor-objects}. This fibre product does indeed coincide with the \'etale sheaf of \Cref{def:stack}, which is by default an algebraic space locally of finite type over $U$ \cite[\href{https://stacks.math.columbia.edu/tag/04XR}{Tag 04XR}]{stacks-project}.

\begin{remark}\label{rmk:alg-space}
A Deligne--Mumford stack $\CX$ such that $\CI_\CX \to \CX$ is an equivalence is called an \emph{algebraic space}. Such notion does not play a role in this paper (except in \Cref{sec:G-groups}, where we defined Grothendieck rings), and neither does the notion of Deligne--Mumford stack. Indeed, Deligne--Mumford stacks have \emph{finite stabilisers}, but automorphism groups of coherent sheaves are never finite. Algebraic spaces are even more special: as fibred categories, they are fibred in \emph{sets}, which makes them much closer to schemes than to actual stacks. We point out that we have full embeddings of categories
\[
\textrm{Schemes}\,\,\,\subset\,\,\,\textrm{Algebraic spaces}\,\,\,\subset\,\,\,\textrm{Deligne--Mumford stacks}\,\,\,\subset\,\,\,\textrm{Algebraic stacks}.
\]
\end{remark}

\subsection{Points of algebraic stacks and stabiliser preserving morphisms}
Let $\CX$ be an algebraic stack. We have the following equivalence relation on the set of morphisms $\Spec K \to \CX$ where $K$ is a field: if $K$ and $K'$ are fields, we say that $p \colon \Spec K \to \CX$ and $p'\colon \Spec K' \to \CX$ are equivalent if there is another field $\Omega$ and a diagram
\[
\begin{tikzcd}
    \Spec \Omega \arrow{r}\arrow{d} & \Spec K \arrow{d}{p} \\
    \Spec K' \arrow{r}{p'} & \CX
\end{tikzcd}
\]
which is 2-commutative. The underlying set of \emph{points} of an algebraic stack $\CX$ is the quotient $\lvert\CX\rvert$ by this equivalence relation. A morphism of algebraic stacks $f \colon \CX \to \CY$ induces a natural group homomorphism
\begin{equation}\label{eqn:isom-map}
\begin{tikzcd}
f_\ast(x) \colon \Iso_\CX(x,x) \arrow{r} & \Iso_\CY(f(x),f(x))
\end{tikzcd}
\end{equation}
for all $x \in \lvert \CX \rvert$.

\begin{definition}\label{def:stabiliser-preserving}
Let $f \colon \CX \to \CY$ be a morphism of algebraic stacks. We say that 
\begin{itemize}
    \item [\mylabel{stab-pres-i}{(i)}] $f$ is \emph{stabiliser preserving} if the natural morphism $\CI_\CX \to \CX \times_{\CY} \CI_\CY$ is an isomorphism,
    \item [\mylabel{stab-pres-ii}{(ii)}] $f$ \emph{induces an isomorphism on automorphism groups} at $x\in \lvert \CX \rvert$ if \eqref{eqn:isom-map} is an isomorphism.
\end{itemize}
\end{definition}

\begin{lemma}\label{lemma:iso-aut}
Let $f\colon \CX \into \CY$ be a locally closed immersion of algebraic stacks. Then $f$ induces an isomorphism on automorphism groups at every point.
\end{lemma}

\begin{proof}
If $\CX \to \CY$ is a monomorphism of algebraic stacks, then the diagram
\[
\begin{tikzcd}
    \CI_\CX \arrow{r}\arrow{d} & \CX\arrow{d}{f} \\
    \CI_\CY \arrow{r} & \CY
\end{tikzcd}
\]
is 2-cartesian \cite[\href{https://stacks.math.columbia.edu/tag/06R5}{Tag 06R5}]{stacks-project}. Locally closed immersions are monomorphisms, thus $\CI_\CX \to \CX \times_{\CY} \CI_\CY$ is an isomorphism, i.e.~Condition \ref{stab-pres-i} in \Cref{def:stabiliser-preserving} is fulfilled. On the other hand, $f$ is unramified, thus being stabiliser preserving is equivalent to inducing  isomorphisms on automorphism groups for every point $x \in \lvert \CX \rvert$ by \cite[\href{https://stacks.math.columbia.edu/tag/0DU9}{Tag 0DU9}]{stacks-project}. 
\end{proof}

%%%%%%%%%%%%%%%%%%%%%%%%%%%%%%%%%%% BIB
\bibliographystyle{amsplain-nodash} 
\bibliography{bib}

\bigskip
\noindent
{\small Andrea T. Ricolfi \\
\address{SISSA, Via Bonomea 265, 34136, Trieste (Italy)} \\
\href{mailto:aricolfi@sissa.it}{\texttt{aricolfi@sissa.it}}
}

\bigskip
\noindent
{\small Barbara Fantechi \\
\address{SISSA, Via Bonomea 265, 34136, Trieste (Italy)} \\
\href{mailto:aricolfi@sissa.it}{\texttt{fantechi@sissa.it}}
}
\end{document}